\documentclass[11pt,reqno,tbtags,a4paper]{amsart}
\usepackage{amssymb}
\usepackage{xpunctuate}
\usepackage{url}
\usepackage[square,numbers]{natbib}
\usepackage{xcolor}
\bibpunct[, ]{[}{]}{;}{n}{,}{,}

\title[The generalized Alice HH vs Bob HT problem]
{The generalized Alice HH vs Bob HT problem}
\date{March 24, 2025}
\author{Svante Janson}
\thanks{SJ is supported by the Knut and Alice Wallenberg Foundation
and
the Swedish Research Council
}
\address{\text{SJ:} Department of Mathematics, Uppsala University, PO Box 480,
SE-751~06 Uppsala, Sweden}
\email{svante.janson@math.uu.se}
\newcommand\urladdrx[1]{{\urladdr{\def~{{\tiny$\sim$}}#1}}}
\urladdrx{http://www2.math.uu.se/~svantejs/papers}
\author{Mihai Nica}
\thanks{MN is supported by the National Sciences and Engineering Research Council of Canada}
\address{\text{MN:} Department of Mathematics and Statistics, University of Guelph, 50 Stone Road E, Guelph, ON N1G 2W1, Canada}
\email{nicam@uoguelph.ca}
\urladdrx{https://nicam.uoguelph.ca/}
\author{Simon Segert}
\address{\text{SS:} New York, NY, USA}
\email{simonsegert@gmail.com}

\subjclass[2020]{} 

\numberwithin{equation}{section}

\renewcommand\le{\leqslant}
\renewcommand\ge{\geqslant}
\renewcommand\leq{\leqslant}
\renewcommand\geq{\geqslant}

\allowdisplaybreaks

\theoremstyle{plain}
\newtheorem{theorem}{Theorem}[section]
\newtheorem{lemma}[theorem]{Lemma}
\newtheorem{proposition}[theorem]{Proposition}
\newtheorem{corollary}[theorem]{Corollary}

\theoremstyle{definition}

\newcommand\xqed[1]{%
    \leavevmode\unskip\penalty9999 \hbox{}\nobreak\hfill
    \quad\hbox{#1}}

\newtheorem{exampleqqq}[theorem]{Example}
\newenvironment{example}{\begin{exampleqqq}}
  {\xqed{$\triangle$}\end{exampleqqq}}

\newtheorem{remarkqqq}[theorem]{Remark}
\newenvironment{remark}{\begin{remarkqqq}}
  {\xqed{$\triangle$}\end{remarkqqq}}

\theoremstyle{remark}

\newcounter{dummy}
\makeatletter
\newcommand\myitem[1][]{\item[#1]\refstepcounter{dummy}\def\@currentlabel{#1}}
\makeatother

\newenvironment{romenumerate}[1][-10pt]{
\addtolength{\leftmargini}{#1}\begin{enumerate}
 \renewcommand{\labelenumi}{\textup{(\roman{enumi})}}%
 }{\end{enumerate}}

\newenvironment{alphenumerate}[1][-10pt]{
\addtolength{\leftmargini}{#1}\begin{enumerate}
 \renewcommand{\labelenumi}{\textup{(\alph{enumi})}}%
 }{\end{enumerate}}

\newcounter{oldenumi}
{\setcounter{oldenumi}{\value{enumi}}
\begin{romenumerate} \setcounter{enumi}{\value{oldenumi}}}
{\end{romenumerate}}

\newcounter{thmenumerate}

\newcounter{xenumerate}   

\newcommand{\refT}[1]{Theorem~\ref{#1}}
\newcommand{\refTs}[1]{Theorems~\ref{#1}}
\newcommand{\refC}[1]{Corollary~\ref{#1}}

\newcommand{\refL}[1]{Lemma~\ref{#1}}
\newcommand{\refLs}[1]{Lemmas~\ref{#1}}
\newcommand{\refR}[1]{Remark~\ref{#1}}

\newcommand{\refS}[1]{Section~\ref{#1}}
\newcommand{\refSs}[1]{Sections~\ref{#1}}
\newcommand{\refSS}[1]{Section~\ref{#1}}
\newcommand{\refSSS}[1]{Section~\ref{#1}}
\newcommand{\refProp}[1]{Proposition~\ref{#1}}

\newcommand{\refE}[1]{Example~\ref{#1}}
\newcommand{\refEs}[1]{Examples~\ref{#1}}

\newcommand{\refApp}[1]{Appendix~\ref{#1}}

\begingroup
  \divide\count255 by 60
  \multiply\count255 by -60
  \advance\count255 by \time
  \ifnum \count255 < 10 \xdef\klockan{\the\count1.0\the\count255}
  \else\xdef\klockan{\the\count1.\the\count255}\fi
\endgroup

\newcommand{\sumkn}{\sum_{k=1}^n}

\newcommand{\prodkn}{\prod_{k=1}^n}

\newcommand\set[1]{\ensuremath{\{#1\}}}
\newcommand\bigset[1]{\ensuremath{\bigl\{#1\bigr\}}}

\newcommand\xpar[1]{(#1)}
\newcommand\bigpar[1]{\bigl(#1\bigr)}
\newcommand\Bigpar[1]{\Bigl(#1\Bigr)}

\newcommand\bigsqpar[1]{\bigl[#1\bigr]}
\newcommand\sqpar[1]{[#1]}

\newcommand\bigabs[1]{\bigl\lvert#1\bigr\rvert}

\newcommand\lrabs[1]{\left\lvert#1\right\rvert}
\def\rompar(#1){\textup(#1\textup)}    
\newcommand\xfrac[2]{#1/#2}

\newcommand\xqfrac[2]{#1/(#2)}

\def\xexp(#1){e^{#1}}

\newcommand\floor[1]{\lfloor#1\rfloor}

\newcommand\ntoo{\ensuremath{{n\to\infty}}}

\newcommand\norm[1]{\lVert#1\rVert}

\newcommand\upto{\nearrow}
\newcommand\punkt{\xperiod}    
\newcommand\iid{i.i.d\punkt}    

\newcommand\eg{e.g\punkt}

\newcommand\ii{\mathrm{i}}

\newcommand{\tend}{\longrightarrow}
\newcommand\dto{\overset{\mathrm{d}}{\tend}}

\newcommand\bbR{\mathbb R}
\newcommand\bbC{\mathbb C}
\newcommand\bbN{\mathbb N}

\newcommand\bbZ{\mathbb Z}

\newcounter{CC}
\newcounter{cc}

\newcommand\E{\operatorname{\mathbb E}{}} 
\renewcommand\P{\operatorname{\mathbb P{}}}

\newcommand\Var{\operatorname{Var}}

\newcommand\Bin{\operatorname{Bin}}

\newcommand\Tr{\operatorname{Tr}}
\newcommand\ran{\operatorname{ran}}

\newcommand\gd{\delta}
\newcommand\gD{\Delta}
\newcommand\gf{\varphi}
\newcommand\gam{\gamma}

\newcommand\gk{\varkappa}
\newcommand\kk{\kappa}
\newcommand\gl{\lambda}
\newcommand\gL{\Lambda}
\newcommand\go{\omega}

\newcommand\gs{\sigma}

\newcommand\gss{\sigma^2}

\newcommand\gth{\vartheta}

\newcommand\eps{\varepsilon}
\renewcommand\phi{\xxx}  

\newcommand\cA{\mathcal A}

\newcommand\cE{\mathcal E}

\newcommand\cQ{\mathcal Q}

\newcommand\cW{\mathcal W}

\newcommand\tP{\widetilde P}

\newcommand\tS{{\widetilde S}}

\newcommand\indic[1]{\boldsymbol1_{#1}} 
\newcommand\one{\boldsymbol1}

\newcommand\qw{^{-1}}

\newcommand\qqw{^{-1/2}}

\newcommand\dd{\,\mathrm{d}}
\newcommand\ddx{\mathrm{d}}

\newcommand{\chf}{characteristic function}

\newcommand\lhs{left-hand side}
\newcommand\rhs{right-hand side}

\newcommand\xoo{_1^\infty}

\newcommand\hP{\widehat P}
\newcommand\hQ{\widehat Q}
\newcommand\hW{\widehat W}
\newcommand\hcW{\widehat {\cW}}
\newcommand\rx{\tilde r}
\newcommand\tr{^{\mathsf t}}
\newcommand\gx{^{\mathsf g}}
\newcommand\sH{\mathsf{H}}
\newcommand\sT{\mathsf{T}}

\newcommand\hS{\widehat S}
\newcommand\xgss{\widetilde\sigma^2}
\newcommand\xmu{\widetilde\mu}

\newcommand\diag{\operatorname{Diag}}
\newcommand\QX{\Pi}
\newcommand\ba{\bar a}

\begin{document}

\begin{abstract} 
In 2024, Daniel Litt posed a simple coinflip game pitting Alice's
  ``Heads-Heads'' vs Bob's ``Heads-Tails'': who is more likely to win if
  they score 1 point per occurrence of their substring in a sequence of $n$
  fair coinflips? This attracted over 1 million views on X 
and  quickly spawned several articles explaining the counterintuitive solution.
We study the generalized
  game, where the set of coin outcomes, $\{ \text{Heads}, \text{Tails} \}$,
  is generalized to an arbitrary finite alphabet $\cA$, and where Alice's and
  Bob's substrings are any finite $\cA$-strings
  of the same length. 
  We find that the winner of
  Litt's game can be determined by a single quantity which measures the
  amount of prefix/suffix self-overlaps in each string; whoever's string has
  more overlaps \emph{loses}. For example, ``Heads-Tails'' beats
  ``Heads-Heads'' in the original problem because ``Heads-Heads'' has a
  prefix/suffix overlap of length 1 while ``Heads-Tails'' has none. The
  method of proof is to develop a precise Edgeworth expansion for discrete
  Markov chains, and apply this to calculate Alice's and Bob's probability
  to win the game correct to order $O(1/n)$.
 \end{abstract}

\maketitle

\section{Introduction}\label{S:intro}

On March 16, 2024 Daniel Litt posted the following brainteaser on 
X \cite{LittX}:
\begin{quote}
Flip a fair coin 100 times—it gives a sequence of heads (H) and tails
(T). For each HH in the sequence of flips, Alice gets a point; for each HT,
Bob does, so e.g.\ for the sequence THHHT Alice gets 2 points and Bob gets 1
point. Who is most likely to win? 
\end{quote}
The post gained popularity due to the deceptively unintuitive nature of the
problem, with plurality of respondents in an attached poll believing
(incorrectly) that the game is fair, and the correct answer being chosen the least often. 

This problem turned out to have a surprisingly rich mathematical
structure. In short order, a succession 
of papers appeared:
\citet{Zeilberger}, \citet{Segert}, and \citet{Grimmett}, 
each of which rigorously analyzed the problem via different
techniques (respectively, symbolic computation, asymptotic analysis, and
probability theory). We highlight here two key results. Firstly, 
Litt's qualitative question ``Who is more likely to win''
was answered for an arbitrary number $n$ of coin flips:
Bob is strictly favored as long as $n$ 
is at least 3 \cite{Segert,Grimmett}. 
Secondly, 
and more immediately relevant to the present work, 
for the quantitative question of finding the winning probabilities, 
the following asymptotic as $\ntoo$ was found
by \cite{Zeilberger,Segert,Grimmett}
(by different methods):
\begin{align}\label{grg1}
  \P(\text{\rm Bob wins}) -\P(\text{\rm Alice wins}) &
\sim 
\frac{1}{2\sqrt{\pi n}}
\end{align}
and, more precisely,
\begin{align}\label{grg2}
\frac12-  \P(\text{\rm Alice wins}) 
\sim 
\frac{3}{4\sqrt{\pi n}},
\qquad
\frac12-  \P(\text{\rm Bob wins}) 
\sim 
\frac{1}{4\sqrt{\pi n}}.
\end{align}

A natural generalization of Litt's problem,
consider for example by \cite{Zeilberger}, 
is to let Alice's and Bob's
strings be arbitrary finite sequences of the same length
$\ell$ in an arbitrary finite alphabet $\cA$; we then assume that 
Alice and Bob do a generalized coin tossing where a sequence of random
letters are generated, with the letters independent and uniformly
distributed.
(See \refS{SAB} for details.
We leave extensions to non-uniform letter distributions to the reader.)
We call also this generalization
\emph{Litt's game}.

This generalization to arbitrary strings was considered by \citet{Basdevant}
who
established, by a combinatorial argument, a condition under which such a
game is exactly fair (see \refR{Rfair} below). 
The main purpose of the present paper is to
establish asymptotics extending \eqref{grg1}--\eqref{grg2}
for the generalized problem. 

Somewhat surprisingly, with the
exception of a small set of pathological strings, we find that who has the
advantage in the generalized Litt's game depends \emph{only} on how
Alice's/Bob's string overlap themselves in their prefixs/suffixes. More
precisely, for a string $A=a_1 a_2 \ldots a_\ell$ of length $\ell$ from an
alphabet $\cA$ with $q$ letters, one can calculate a single quantity
$\theta_{AA} \in \mathbb{R}$ that measures the size of the prefix/suffix
overlap of $A$ defined as follows:
\begin{align}\label{th_def}
\Theta(A,A)&:=\set{ 1\le k\le \ell-1: a_{\ell-k+1}\dotsm a_\ell = a_1\dotsm a_k},
\\\label{thetaAA}
\theta_{AA}&:
=\sum_{k\in\Theta(A,A)} q^{k-\ell}
=q^{-\ell}\sum_{k\in\Theta(A,A)} q^{k}.
\end{align}
(This is a special case of $\theta_{UV}$ between two strings $U,V$
  defined in \eqref{tau}.) 
  We will prove the following result, which shows
  that the asymptotic winner of Litt's game is determined by comparing the
  value $\theta_{AA}$ of Alice's string to $\theta_{BB}$ of Bob's string;
  whoever's $\theta$ overlap value is larger \emph{loses} Litt's game
  asymptotically. 

\begin{theorem}\label{TAB}
Let Alice and Bob play Litt's game  with distinct words $A$ and $B$ of the
same length $\ell$ in an alphabet $\cA$ with $q$ letters, and assume that
$n$ letters are chosen  at random, uniformly and independently.

Exclude the two cases, both with $q=2$: 
\begin{romenumerate}
\item\label{TABa} 
$A=\sH\sT^{\ell-1}$ and $B=\sT^{\ell-1}\sH$ for some $\ell\ge2$
(see \refE{Egss2}),
\item \label{TABc}
$A=\sH$ and $B=\sT$ 
(see \refE{EH-T}),
\end{romenumerate}
and their variants obtained by interchanging Alice and Bob 
or $\sH$ and $\sT$ (or  both).

Then, with $\theta_{UU}$ given by \eqref{thetaAA} (or \eqref{tau}) and $\gss = \gss(A,B)$ given by \eqref{gss9}, we have $\gss > 0$, 
and 
\begin{align}\label{tab1}
\P(\text{\rm Alice wins}) &
= \frac12+\frac{\theta_{BB}-\theta_{AA}-1}{2\sqrt{2\pi\gss}}n\qqw+O(n\qw),
\\\label{tab2}
\P(\text{\rm Bob wins}) &
= \frac12+\frac{\theta_{AA}-\theta_{BB}-1}{2\sqrt{2\pi\gss}}n\qqw+O(n\qw),
\\\label{tab3}
\P(\text{\rm Tie}) &
= \frac{1}{\sqrt{2\pi\gss}}n\qqw+O(n\qw),
\end{align}
and thus
\begin{align}\label{tab4}
\P(\text{\rm Alice wins}) 
-
\P(\text{\rm Bob wins}) &
= \frac{\theta_{BB}-\theta_{AA}}{\sqrt{2\pi\gss}}n\qqw+O(n\qw).
\end{align}
\end{theorem}

In the two excluded cases, the conclusions \eqref{tab1}--\eqref{tab3} fail
(in somewhat different ways), 
see \refEs{Egss2} and \ref{EH-T}.

\begin{remark}\label{Rfair}
In particular, \refT{TAB} shows that 
the game is fair up to order $n\qw$ 
if and only if $\theta_{AA}=\theta_{BB}$. This is (in different notation)
the condition by \citet{Basdevant},
who showed (by completely different methods)
that in this case the game is perfectly fair for every $n$:
$\P(\text{\rm Alice wins}) =
\P(\text{\rm Bob wins})$.
(The result in \cite{Basdevant} is stated for $q=2$, but the proof holds for
an arbitrary finite alphabet.)
Our result thus shows that if the condition in \cite{Basdevant} is
not satisfied, then the game is for all large $n$ not fair; thus their
condition for fairness is both necessary and sufficient.
\end{remark}

Our method
to prove \refT{TAB} is to 
recognize \eqref{tab1} and \eqref{tab2} as examples of (first order)
Edgeworth expansions.
Edgeworth expansions are useful approximations in many situations in
probability and statistics, and they have been rigorously established in
many situations, see \refSS{SSEdgeworth} for a background.
We use  a general result giving an Edgeworth
expansion for partial sums of an integer-valued function of a finite-state
Markov chain.
There is a large literature on Edgeworth expansions in various situations, 
including Markov chains, see \refSS{SSEdgeworth}.
However,  we have failed to find a general theorem that is
directly applicable here.
We therefore state such a theorem for finite-state Markov chains here
(in two versions, \refTs{T1} and \ref{T2}); it is perhaps not new, 
but since we do not know a reference we give for completeness a complete
proof.

\section{Preliminaries}

\subsection{Notation}


The \emph{characteristic function} of a random variable $X$ is defined by
\begin{align}\label{chf}
\gf(t)=  \gf_X(t):=\E e^{\ii tX},\qquad t\in\bbR.
\end{align}

Vectors are generally regarded as column vectors in formulas using matrix
notation; the row vector that is the transpose of a column vector $v$ 
is denoted $v\tr$. 

We let $\one=(1,\dots,1)\tr$, the (column) vector with all entries 1, with the dimension of the vector determined by context. 

For vectors $v$ (in $\bbC^m$ for some $m\in\bbN$) we let $\norm{v}$ be the
usual Euclidean norm.
For a square matrix $A$, let
$\norm{A}$ denote its \emph{operator norm}
\begin{align}\label{norm}
  \norm {A} :=\sup\bigset{\norm{Av}:\norm{v}=1} 
\end{align}
and let $\rho(A)$ denote its \emph{spectral norm}
\begin{align}\label{rho}
  \rho(A):=\max\set{|\gl|:\gl\text{ is an eigenvalue of $A$}}
.\end{align}
Recall the \emph{spectral radius formula}
\begin{align}\label{spec}
  \rho(A) = \lim_\ntoo \norm{A^n}^{1/n}.
\end{align}

A square matrix $A$ with non-negative entries is 
\emph{irreducible} if for every pair of indices $i,j$ we have $(A^n)_{ij}>0$
for some $n\ge1$; the matrix is \emph{primitive} 
if we further can choose the same $n$ for all pairs $i,j$;
see e.g.\ \cite[Chapter 1]{Seneta}.
Recall that a matrix is primitive if and only if it is irreducible and
\emph{aperiodic} 
\cite[Theorem 1.4]{Seneta}.
A \emph{stochastic} matrix is a square matrix with non-negative entries
where all row sums are 1.
It follows from the Perron--Frobenius theorem \cite[Theorem 1.5]{Seneta}
that an irreducible stochastic matrix has an eigenvalue 1 which is
(algebraically) simple, and that the spectral radius is 1;
corresponding right and left eigenvectors are $\one$ and $\pi\tr$, the
stationary distribution of the corresponding Markov chain.

For $z_0\in\bbC$ and $r>0$, let $D(z_0,r):=\set{z\in\bbC:|z-z_0|<r}$,
the open disc with centre $z_0$ and radius $r$.

We let $\indic{\cE}$ denote the indicator of an event $\cE$;
this is thus 1 if $\cE$ occurs and 0 otherwise.

$C$ denotes unspecified constants that may vary from one occurrence to the
next. 
They may depend on parameters,
and we may write e.g.\ $C(\gd)$ to stress this.


\subsection{Cumulants}

If $X$ is a random variable, with (for simplicity) all moments finite,
then its \emph{cumulants} (also called \emph{semi-invariants})
are defined by, for integers $m\ge1$,
recalling \eqref{chf},
\begin{align}\label{cu2}
\kk_m=  \kk_m(X):=\ii^{-m}\frac{\ddx^m}{\ddx t^m} \log \gf_X(t)\bigr|_{t=0};
\end{align}
in other words, $\log\gf_X(t)$ has the Taylor expansion, for $M\ge1$
and small $t$,
\begin{align}\label{cu3}
  \log\gf_X(t)=\sum_{m=1}^M\kk_m\frac{(\ii t)^m}{m!}+O(|t|^{M+1}).
\end{align}
See for example \cite[Section 4.6]{Gut}.
The cumulant $\kk_m$ can be expressed as a polynomial in moments of order
at most $m$. In particular \cite[Theorem 4.6.4]{Gut},
\begin{align}\label{cu4}
  \kk_1 &= \E X,
\\\label{cu5}
\kk_2&=\E[X^2]-(\E X)^2 = \Var(X),\\
\kk_3&=\E[(X-\E X)^3].\label{cu6}
\end{align}
Thus $\kk_2$ and $\kk_3$ are simply the corresponding central moments. 
(Higher cumulants have more complicated formulas.)

\subsection{Background on Edgeworth expansions}\label{SSEdgeworth}
The idea of an Edgeworth expansion (or Edgeworth approximation)
\cite{Edgeworth}
is that many random variables in probability theory and statistics have
distributions that are approximatively normal, and that the approximation
often can be improved by adding extra terms to the normal distribution.
It turns out that the natural terms to add are derivatives of the normal
density function, with coefficients that are given by some (explicit but a
little complicated) polynomials of cumulants $\kk_m$, $m\ge3$,
divided by powers of the variance.
(Part of the motivation is that for a normal random variable, all
cumulants $\kk_m=0$ for $m\ge3$; thus the cumulants measure in some way
deviations from normality.
For detailed motivations, see \cite{Edgeworth} and
\cite[Chapter 17.6-7]{Cramer}.)
For a continuous random variable $Z$, 
for simplicity normalized to have $\E Z=0$ and $\Var Z=1$,
the one-term Edgeworth appoximation is
\begin{align}\label{ed1}
\P(Z\le x) 
\approx \Phi(x) + \frac{\kk_3(Z)}{6\sqrt{2\pi}}(1-x^2)e^{-x^2/2},
\qquad -\infty<x<\infty,
\end{align}
where 
\begin{align}
  \label{Phi}
\Phi(x):=\int_{-\infty}^x\frac{1}{\sqrt{2\pi}}e^{-x^2/2}\dd x
\end{align}
is the standard normal distribution function.
This approximation can be made rigorous, with error bounds, in 
many situations. The most important, and archetypical, case is when the
random variable $Z$ is 
a normalized sum $\tS_n:=\frac{1}{\gs\sqrt n}\sum_1^n X_i$ of $n$ 
independent and identially distributed (i.i.d.)\
random variables $X_i$ with $\E X_i=0$, 
$\Var X_i=\gss$, and a finite third moment $\E |X_i|^3$.
Then $\kk_3(\tS_n)=n^{-1/2}\kk_3(X_i)/\gs^3$, and, as shown by 
\cite[Theorem IV.2, p.~49]{Esseen}, 
if the distribution of $X_i$ is non-lattice, then
\eqref{ed1} holds with an error $o(n^{-1/2})$,
i.e., 
with $S_n:=\sum_1^n X_i$ the unnormalized sum,
\begin{align}\label{ed2}
\P(S_n\le x\gs\sqrt n)=
\P(\tS_n\le x) = 
\Phi(x) + \frac{\kk_3(X_1)}{6\gs^3\sqrt{2\pi n}}(1-x^2)e^{-x^2/2}
+ o\bigpar{n^{-1/2}},
\end{align}
uniformly in $x$.
(Note that the \rhs{} is $\Phi(x)+O(n\qqw)$ uniformly in $x$, so \eqref{ed2}
implies, and can be seen as a precise version of, the Berry--Esseen theorem 
\cite[Theorem 7.6.1]{Gut}.)
Furthermore, under somewhat stronger conditions, the expansion \eqref{ed2}
can be continued to any number of terms, where the $m$th term is of the
order $n^{-m/2}$; see \cite[Theorem IV.1, p.~48]{Esseen} or
\cite[Theorem VI.4 and (VI.1.13)]{Petrov} for a precise statement and the
general form of the terms.

However, in this paper we are interested in integer-valued variables, and
then \eqref{ed2} is not appropriate, since the \rhs{}
ignores the jumps in the \lhs{} when $x\gs\sqrt n$ is an integer.
The correct version of \eqref{ed2} for integer-valued $X_i$ with $\E X_i=0$
is
\begin{multline}\label{ed3}
\P(S_n\le x\gs\sqrt n)
=  \P(\tS_n\le x) 
= \Phi(x) + \frac{\kk_3(X_1)}{6\gs^3\sqrt{2\pi n}}(1-x^2)e^{-x^2/2}
\\
+\frac{1}{\gs\sqrt{2\pi n}}\gth(x\gs\sqrt n)e^{-x^2/2}
+ o\bigpar{n^{-1/2}},
\end{multline}
where we have added a correction term with
\begin{align}\label{gth}
  \gth(x):=\tfrac12 - (x-\floor{x}),
\end{align}
see \cite[Theorem IV.3, p.~56]{Esseen} (which also includes the case 
$\E X_i\neq0$, for simplicity ignored here).
See also 
\cite[Theorem IV.4, p.~61]{Esseen} and \cite[Theorem VI.6]{Petrov}
for the corresponding result with further terms (similar, but more complicated),
and \cite{KolassaMc} for an interpretation of this expansion as a Sheppard's
correction of the version for the continuous 
(or, more generally, non-lattice) case.

Note that $\gth(x)=0$ when $x$ is  the midpoint between two consecutive
integers. In fact, 
\eqref{ed3} can be regarded as an approximation as in \eqref{ed2} 
of the distribution function of $S_n$ interpolated linearly between
such half-integer points, see \cite[Theorem XVI.4.2]{FellerII}.

In the present paper we are interested in the distribution function at or
close to the mean, i.e., the case $x\approx0$ above. Note that if, say,
$x\gs\sqrt n=O(1)$, so $x=O(n^{-1/2})$,
then \eqref{ed3} simplifies to 
\begin{align}\label{ed4}
&\P(S_n\le x\gs\sqrt n)
= \frac12 +\frac{1}{\sqrt{2\pi}}x + \frac{\kk_3(X_1)}{6\gs^3\sqrt{2\pi n}}
+\frac{1}{\gs\sqrt{2\pi n}}\gth(x\gs\sqrt n)
+ O\bigpar{n^{-1}}.
\end{align}

We remark that under suitable conditions there are also corresponding
Edgeworth expansions of the density function (in the absolutely
continuous case) or point probabilities (in the discrete case), which refine
the local limit theorem in a similar way, see for example
\cite[Theorem IV.5, p.~63]{Esseen} or
\cite[Theorem VII.13]{Petrov}; 
we will not use these expansions here.

In the present paper, we cannot use the basic results above, since the random
variable we are interested in is \emph{not} a sum of i.i.d.\ variables.
Nevertheless, the Edgeworth expansions above have been extended to various
other situations, and
there is a large literature on Edgeworth expansions under various situations
with dependency;
we mention only a few references that are closely related to our case,
although none of them is directly applicable to it,
see further the discussion before \refT{T1}.
We consider below an irreducible finite-state Markov chain, and random
variables defined by \eqref{jp1}. For this case, 
\citet{Sirazhdinov} proved (under some conditions) 
an Edgeworth expansion for point probabilities
(and a Berry--Esseen type estimate for the distribution function).
\citet{Nagaev1961} considered more general Markov chains, allowing an infinite
number of states, and proved Edgeworth expansions for both continuous and
discrete cases under some conditions.
(For the continuous case, see also \cite{Datta} 
with some corrections and improvements.) 
For the integer-valued case that we are interested in,
\citet[Theorem (3.1)]{Hipp}  
proved an Edgeworth expansion for point probabilities 
under more general conditions.
The variables in our application to Litt's game are $m$-dependent, so an
alternative approach would be to use results on Edgeworth expansions for sums of
$m$-dependent variables. However, we have not found a suitable such theorem
for integer-valued variables; for the non-lattice case, see e.g.\
\citet{Heinrich1984}, \citet{Rhee}, and \cite{Loh}; 
see also \citet{RinottRotar} for a more
general result, and the further references there.

\subsection{Group Inverse}
If $A$ is a square (possibly singular) 
matrix, then the \emph{group inverse} 
$A\gx$ is defined to be the matrix satisfying: 
\begin{align}
AA\gx&=A\gx A,\label{gi1}\\
AA\gx A&=A,\label{gi2}\\
A\gx AA\gx&=A\gx,\label{gi3}
\end{align}
The matrix $A\gx$ is unique if it exists, however it may not exist in the
first place. 

If we regard $A$ as a linear operator in some or $\bbR^q$ (or $\bbC^q$),
with kernel $\ker(A)$ and range $\ran(A)$, 
then $A\gx$ exists if and only if  $\ker(A)\cap\ran(A)=\set0$ and thus
$\bbR^q$ (resp.\ $\bbC^q$) is the direct sum $\ker(A)\oplus\ran(A)$; in this case
$A\gx|_{\ker(A)}=0$ and $A\gx|_{\ran(A)}$ is the inverse of $A:\ran(A)\to\ran(A)$. 

It can be shown that $A\gx$ does exist whenever $A=I-P$ where $P$ is a
stochastic  matrix, see \cite{meyer}. 
In this case, there is actually more that we can say. One useful identity
that holds if $P$ is irreducible is 
\begin{equation}\label{gi4}
(I-P)\gx(I-P)=I-\one\pi\tr
\end{equation}
where $\pi$ is the stationary distribution,
see Theorem 2.2 in \cite{meyer}.

Most of the time when we want to actually compute $(I-P)\gx$, we will use the following well-known representation. 
\begin{proposition}
\label{groupinvformulaprop}
If\/ $P$ is an irreducible stochastic matrix, then 
\begin{eqnarray}
(I-P)\gx&=&(I-P+\one\pi\tr)^{-1}-\one\pi\tr\notag\\ 
&=&\lim_{t\upto 1}\bigpar{(I-tP)^{-1}-\one\pi\tr/(1-t)}
\label{gi6}
.\end{eqnarray}
\end{proposition}
\begin{proof}
The first equality is proved in \cite{meyer}. For the second equality, write 
\begin{eqnarray}\label{gi7}
(I-P+\one\pi\tr)^{-1}&= & \lim_{t\upto 1}(I-t(P-\one\pi\tr))^{-1}
\end{eqnarray}
Since for any $t<1$, the spectral radius of $t(P-\one\pi\tr)$ is $<1$, the
inverse on the right side can be expanded as a geometric series. Moreover,
it is easily seen that $(P-\one\pi\tr)^n=P^n-\one\pi\tr$ for $n\geq 1$,
which implies the second equality after a simple calculation.
\end{proof}

\begin{remark}\label{Rpseudoinverse}
  Note that the group inverse does not in general coincide with the more
  well-known Moore--Penrose pseudoinverse $A^+$. In fact, it can be shown
  that if $P$ is an irreducible stochastic matrix, then
$(I-P)\gx=(I-P)^+$ if and only if the stationary distribution of $P$
  is uniform \cite[Theorem 6.1]{meyer}.
\end{remark}

\section{An Edgeworth expansion for finite-state Markov chains}
\label{SMarkov}
\subsection{Markov chain preliminaries}
Let $(W_k)\xoo$ be a homogeneous Markov chain on a finite state space $\cW$,
with transition probabilities given by the matrix
$P=(P_{ij})_{i,j\in\cW}$.
(All matrices and vectors in this section will be indexed by $\cW$.
The reader that so prefers may without loss of generality assume that
$\cW=[m]$ for some integer $m$ in this section.)
For basic facts on Markov chains used below, see \eg{}
\cite[Chapter 1]{Norris}.

We denote the distribution of $W_k$ by 
\begin{align}\label{jp2}
  \pi_{k}=(\pi_{k;i})_{i\in\cW}
\quad\text{where}\quad
\pi_{k;i}:=\P(W_k=i),
\quad i\in\cW.
\end{align}
In particular, $\pi_1$ is the initial distribution of the Markov chain.
We regard $\pi_k$ as a column vector;  recall that
its transpose (a row vector) is denoted by $\pi_k\tr$.
It is well-known that
\begin{align}\label{jp3}
  \pi_k\tr = \pi_1\tr P^{k-1},
\qquad k\ge1.
\end{align}

We say that a sequence $i_0,\dots, i_\ell$ of elements of $\cW$ is a
\emph{path} if it can appear with positive probability in the Markov chain,
i.e., if $P_{i_{k}i_{k+1}}>0$ for $0\le k<\ell$. We say that $\ell$ is the
\emph{length} of the path; we  denote the length of a path $\cQ$ by $\ell(\cQ)$.
A \emph{closed path} is a path
$i_0,\dots, i_\ell$ such that $i_\ell=i_0$.
We assume throughout this section
that the Markov chain is \emph{irreducible}, i.e., that
\begin{align}\label{em0}
\text{for every pair $i,j\in\cW$, there exists a path $i=i_0,\dots,i_\ell=j$
  from $i$ to $j$}.  
\end{align}
We assume also that the Markov chain is \emph{aperiodic}, i.e.,
\begin{align}\label{em1}
  \gcd\bigset{\ell(\cQ):\cQ \text{ is a closed path}}=1.
\end{align}
It is well-known
that our assumptions \eqref{em0}--\eqref{em1} that the Markov chain is
irreducible and aperiodic imply that
there is a unique stationary
distribution $\pi$, i.e., a distribution
satisfying
\begin{align}\label{jp4}
  \pi\tr=\pi\tr P.
\end{align}
Moreover, for any initial distribution $\pi_1$,
we have
\begin{align}\label{jp5}
  \pi_n\to\pi\qquad\text{as }\ntoo.
\end{align}

\subsection{The main result}

Let $g:\cW\to\bbZ$ be an integer-valued function, and define the 
integer-valued random variables
$X_k:=g(W_k)$ and
\begin{align}\label{jp1}
  S_n:=\sumkn X_k = \sumkn g(W_k).
\end{align}
Our goal is to prove an Edgeworth expansion for $S_n$.

We will, besides the aperiodicity condition \eqref{em1}, also assume 
a similar aperiodicity condition for the function $g$.
For a path $\cQ$ given by $i_0,\dots,i_\ell$, define the 
\emph{value} of $\cQ$ as
\begin{align}\label{em3}
  g(\cQ):=\sum_{k=1}^\ell g(i_k) \in \bbZ.
\end{align}
We then assume:
\begin{align}\label{em4}
\text{The set $\bigset{(g(\cQ),\ell(\cQ)):\cQ \text{ is a closed path}}$
generates $\bbZ^2$ as a group}.
\end{align}

\begin{remark}\label{Rem4}
  It is easy to see, arguing similarly as in the proof of \refProp{PO}
  below,
that in \eqref{em4}, it is equivalent to consider only closed paths starting
at any given state $i_0$. This implies the following probabilistic
formulation of the condition:
Start the Markov chain in state $i_0$ at time 0, and let
$T:=\min\set{k\ge1:W_k=i_0}$ be the time of the first return there.
Then \eqref{em4} is equivalent to:
\begin{align}\label{em4x}
\text{
  The random vector $(S_T,T)$ is not supported on a proper sublattice of 
$\bbZ^2$}.
\end{align}
This formulation is used e.g.\ by \cite{Hipp}.
\end{remark}

The following results will be proved in \refS{SpfT1} .
As noted in \refSS{SSEdgeworth}, 
there is a large literature on Edgeworth expansions, and in particular
\citet{Nagaev1961} has proven several similar results for
more general Markov chains (allowing also infinite state spaces); 
however, his theorems for the integer-valued (or, equivalently, lattice-valued)
case
assume instead of \eqref{em4} 
a strong aperiodicity condition 
(\cite[Condition C]{Nagaev1961}, see \refE{ENagaev})
which is not satisfied in our
application.
\citet[Theorem (3.1)]{Hipp} has, by another method, 
proven a general result on approximation of the point probabilities
using weaker assumptions than \cite{Nagaev1961} and in particular
under the same aperiodicity condition \eqref{em4} as we.
However, his result is not in a form directly applicable to our probabilities;
it seems possible to derive our result from his with some extra work
but we have not attempted this; we give instead a direct proof,
which also leads to the explicit formulas  \eqref{t1a}--\eqref{t1f} 
for the parameters (cumulants)
in the approximation \eqref{t1} (which are not explicit in \cite{Hipp}).
Our proof  is similar to the proofs in
\cite{Nagaev1961}, but much simpler since we consider only finite state spaces;
our argument is thus structurally closer  to the classical proof of the i.i.d.\
result using characteristic functions than the arguments by \cite{Hipp}.

\begin{theorem}\label{T1}
Let $(W_k)\xoo$ be a stationary, irreducible and aperiodic 
homogeneous Markov chain on a finite state space  $\cW$, and
let $S_n$ be defined by \eqref{jp1} for some function
$g:\cW\to\bbZ$ such that \eqref{em4}  holds. 
Let 
\begin{align}\label{t1kk1}
  \mu&:= \E X_1=\E g(W_1), 
\\\label{t1kk2}
\gss &:=\lim_\ntoo n\qw\Var(S_n),
\\\label{t1kk3}
\gk_3&:=\lim_\ntoo n\qw\kk_3(S_n)
.\end{align}
(The limits exist and are finite  under our assumptions.)
Then $\gss>0$,
$\E S_n=n\mu$, and with $\Phi(x)$ and $\gth(x)$ as in 
\eqref{Phi} and \eqref{gth}, 
\begin{multline}\label{t1}
\P(S_n -n\mu\le x\gs\sqrt n)
= \Phi(x) + \frac{\gk_3}{6\gs^3\sqrt{2\pi n}}(1-x^2)e^{-x^2/2}
\\
+\frac{1}{\gs\sqrt{2\pi n}}\gth(x\gs\sqrt n + n\mu)e^{-x^2/2}
+ O\bigpar{n^{-1}},
\end{multline}
uniformly in $x\in\bbR$ and $n\ge1$.
Moreover, $\mu$, $\gss$, and $\gk_3$ have the following explicit expressions: 
\begin{align}
\mu&=\pi\tr G\one \label{t1a}\\
\gss& =  \pi\tr\left(G^2+2GQPG\right)\one-\mu^2\label{t1b}\\
\gk_3&=\pi\tr\left(G^3+3GQPG^2+ 3G^2QPG+6GQPGQPG\right)\one\notag\\
&\qquad-\mu\left(6\pi\tr GQ^2PG\one+ 3\sigma^2+\mu^2\right),\label{t1c}
\end{align}
where $\pi$ is the stationary distribution of $P$ and $G=\diag(g(1),\dots,
g(m))$ and $Q:=(I-P)\gx$.
Equivalently, with $Q':=Q-I = (I-P)\gx-I$,
we have
\begin{align}
\sigma^2&=\pi\tr(G^2+2GQ'G)\one+\mu^2\label{t1e}\\
\gk_3& =  \pi\tr(G^3+3GQ'G^2+3G^2Q'G+6GQ'GQ'G)\one\notag\\
&\qquad+\mu\left(3\pi\tr G^2\one - 6 \pi\tr GQ'^2G\one +2\mu^2\right)
\label{t1f}
.\end{align}

\end{theorem}

The conditions \eqref{em4} and $\gss>0$ are discussed further in
\refSS{SSem4} and \ref{SSgss0}.

In particular, for the case $\mu=0$ and $x=0$ we obtain the following.

\begin{corollary}\label{C1}
Under the assumptions of \refT{T1}, if furthermore $\mu=0$, then
\begin{align}\label{c1+}
\P(S_n \le 0)
= \frac12 + \frac{\gk_3}{6\gs^3\sqrt{2\pi n}}
+\frac{1}{2\gs\sqrt{2\pi n}}
+ O\bigpar{n^{-1}}
\end{align}
and
\begin{align}\label{c1-}
\P(S_n < 0)
= \frac12 + \frac{\gk_3}{6\gs^3\sqrt{2\pi n}}
-\frac{1}{2\gs\sqrt{2\pi n}}
+ O\bigpar{n^{-1}}.
\end{align}
Consequently,
\begin{align}\label{c1+-}
\P(S_n < 0) - \P(S_n>0)
=  \frac{\gk_3}{3\gs^3\sqrt{2\pi n}}
+ O\bigpar{n^{-1}}
\end{align}
and
\begin{align}\label{c10}
\P(S_n = 0)
= \frac{1}{\gs\sqrt{2\pi n}}
+ O\bigpar{n^{-1}}.
\end{align}
\end{corollary}

The local limit theorem \eqref{c10} is here shown as a consequence of
\eqref{c1+}--\eqref{c1-}, and thus of \eqref{t1}.
It can also easily be proved directly from the estimates in \eqref{l4} and
\eqref{l7} below using Fourier inversion;  
furthermore, further terms can be obtained as in \cite[Theorem IV.5]{Esseen},
see \cite[Theorem 2]{Sirazhdinov}.

We can generalize the set-up above and consider a function
$g:\cW\times\cW\to\bbZ$ of two variables; we then define
\begin{align}\label{jb1}
  S_n:=\sumkn g(W_{k-1},W_k),
\end{align}
where we for convenience index the Markov chain as $(W_k)_0^\infty$.
We then similarly define, if $\cQ$ is the path $i_0,\dots,i_\ell$,
\begin{align}\label{em33}
  g(\cQ):=\sum_{k=1}^\ell g(i_{k-1},i_k) \in \bbZ.
\end{align}

\begin{theorem}
  \label{T2}
Let $(W_k)\xoo$ be a stationary, irreducible and aperiodic 
homogeneous Markov chain on a finite state space  $\cW$, and
let $S_n$ be defined by \eqref{jb1} for some function
$g:\cW^2\to\bbZ$ such that \eqref{em4}  holds.
Then the conclusions of \refT{T1} and \refC{C1} hold. The explicit expressions for the moments are modified accordingly, see Section \ref{stateExpPiQ} for details.
\end{theorem}

For further extensions, see \refS{Sext}.

\section{Coin tossing}\label{SAB}

We consider in this section Litt's game, in the general version described in
the introduction which is defined as follows.

Let $\cA$ be a finite alphabet, with $q:=|\cA|\ge2$ letters.
(We are mainly interested in the case $q=2$, and then we let
$\cA=\set{\sH,\sT}$
as in Litt's original formulation with coin tossing.)
Let $\Xi_n=\xi_1\dotsm\xi_n$ be a random word with $n$ letters $\xi_i\in\cA$,
chosen independently and uniformly at random. Hence, for any given word
$a_1\dotsm a_n$, the probability
\begin{align}\label{ju1}
  \P\xpar{\xi_1\dotsm\xi_n=a_1\dotsm a_n}=q^{-n}.
\end{align}
Fix two distinct words $A=a_1\dotsm a_\ell$ and $B:=b_1\dotsm b_\ell$ of the
same length $\ell\ge1$. 
The letters in $\Xi_n$ are drawn one by one;
Alice scores a point when the last $\ell$ letters form $A$,
and Bob scores a point when they form $B$. Denote their total scores by
$S_{A,n}$ and $S_{B,n}$;
we are interested in the difference $\hS_n:=S_{A;n}-S_{B,n}$.

To put this into our framework, define for $1\le k\le n-k+1$ the
subword
\begin{align}\label{ju2}
  W_k:=\xi_k\dotsm\xi_{k+\ell-1} \in\cA^\ell,
\end{align}
and 
for any word $U\in \cA^\ell$,  define the indicator 
\begin{align}\label{ju3}
  I_{U}(V):=\indic{U=V},
\qquad V\in \cA^\ell.
\end{align}
Then the net score of the $(k+\ell-1)$th draw (i.e., the $k$th draw that may
score) is
\begin{align}\label{ju4}
  X_k:=I_A(W_k)-I_B(W_k),
\end{align}
and thus the final net score is
\begin{align}\label{ju5}
  \hS_n:=\sum_{k=1}^{n-\ell+1} X_k.
\end{align}

We may as well assume that the random letter $\xi_i$ is defined for every
$i\ge1$; then also $W_k$, $X_k$, and $S_n$ are defined for all $k\ge1$ and
$n\ge\ell$. It is obvious that the sequence of random words $(W_k)\xoo$
forms a stationary, homogeneous Markov chain with state space $\cA^\ell$;
furthermore, it is easy to see that this chain is irreducible and aperiodic.
By comparing \eqref{ju5} and \eqref{jp1} we see that
\begin{align}\label{ju6}
  \hS_n = S_{n-\ell+1}
\end{align}
for the function  $g:\cA^\ell\to\bbZ$  given by $g:=I_A-I_B$.
Note that for any fixed word $U\in\cA^\ell$ and any $k\ge1$ we have
\begin{align}\label{ju7}
  \E[I_U(W_k)] = \P(W_k=U)=q^{-\ell},
\end{align}
and thus
\begin{align}\label{ju8}
  \mu:=\E g(W_1)
=\E[I_A(W_1)]-\E[I_B(W_1)]
=q^{-\ell}-q^{-\ell}=0.
\end{align}
We are thus in the setting of \refT{T1} and \refC{C1}, 
provided that \eqref{em4} holds.
All that remains is thus to calculate $\gss$ and $\gk_3$ and to verify
\eqref{em4}.

We provide two methods, one a direct method that is specific
  to this application, and a second method that involves the explicit formulas 
\eqref{t1a}--\eqref{t1c}
for the moments. 

We introduce some more notation.
For a  pair of two words $U=u_1\dotsm u_\ell$ and $V=v_1\dotsm v_\ell$
of length $\ell$, let
\begin{align}\label{jv1}
  \Theta(U,V):=\set{ 1\le k\le \ell-1: u_{\ell-k+1}\dotsm u_\ell = v_1\dotsm v_k},
\end{align}
i.e., the set of integers $k<n$ such that the words consisiting of the last
$k$ letters of $U$ and the first $k$ letters of $V$ are the same; in other
words, $U$ and $V$ may be concatenated with an overlap of $k$ letters.
(Note that this is not symmetric: in general $\Theta(U,V)\neq\Theta(V,U)$.)
Note that we do not include $\ell$ in $\Theta(U,V)$, even if $U=V$.
We define also the rational number, 
which may be called \emph{overlap index},
\begin{align}\label{tau}
  \theta_{UV}:
=\sum_{k\in\Theta(U,V)} q^{k-\ell}
=q^{-\ell}\sum_{k\in\Theta(U,V)} q^{k}
.\end{align}
This is equivalent to the quantity $[U|V]$ defined by Basdevant et al.\ \cite{Basdevant}, which plays a key role in their analysis. More precisely, $\theta_{UV}=q^{-\ell}[U|V]$.

For two fixed words $U,V\in\cA^\ell$, and $1\le j\le\ell-1$,
we have $I_U(W_i)I_V(W_{i+j})=0$ unless $\ell-j\in\Theta(U,V)$, and it follows
that, since $W_i$ and $W_{i+j}$ together contain $\ell+j$ random letters,
\begin{align}\label{jv4}
  \E\bigsqpar{I_U(W_i)I_V(W_{i+j})}
=\indic{\ell-j\in\Theta(U,V)}\,q^{-(\ell+j)}.
\end{align}
Hence
\begin{align}\label{jv5}
\sum_{j=1}^{\ell-1}  \E\bigsqpar{I_U(W_i)I_V(W_{i+j})}
&=
\sum_{j=1}^{\ell-1}\indic{\ell-j\in\Theta(U,V)}\,q^{-(\ell+j)}
\notag\\&
=
\sum_{i=1}^{\ell-1}\indic{i\in\Theta(U,V)}\,q^{i-2\ell}
=q^{-\ell}\theta_{UV}.
\end{align}

We return to $S_n$.
We have $\E S_n=n\mu=0$ by \eqref{ju8} and thus $\Var S_n =
\E[S_n^2]$. Furthermore, by the definition \eqref{jp1} 
(or \eqref{ju5}--\eqref{ju6})
and expanding,
\begin{align}\label{jv2}
\E[S_n^2] &
=\sum_{i,j=1}^n\E\sqpar{X_iX_j}
=\sum_{i=1}^n\E\sqpar{X_i^2}
+2\sum_{i=1}^n\sum_{k=1}^{n-i}\E\sqpar{X_iX_{i+k}}.
\end{align}
If $k\ge\ell$, then $W_i$ and $W_{i+k}$ consist of different letters from
$\Xi$, and thus they are independent, which implies
$\E\sqpar{X_iX_{i+k}}=\E\sqpar{X_i}\E[X_{i+k}]=0$;
hence it suffices to take $k<\ell$ in the inner sum in \eqref{jv2}.
Furthermore, it is clear that for any fixed $k\ge0$, 
the expectation $\E\sqpar{X_iX_{i+k}}$ does not
depend on $i$. Hence, \eqref{jv2} yields
\begin{align}\label{jv3}
\E[S_n^2] &
=n\E\sqpar{X_1^2}
+2n\sum_{k=1}^{\ell-1}\E\sqpar{X_1X_{1+k}} + O(1),
\end{align}
where the $O(1)$ comes from the missing terms with $0\le n-i<k<\ell$.
Consequently, \eqref{t1kk2} yields
\begin{align}\label{gss}
\gss=\E\sqpar{X_1^2} +2\sum_{k=1}^{\ell-1}\E\sqpar{X_1X_{1+k}}.
\end{align}
Recalling \eqref{ju4}, we see that
\begin{align}\label{jv66}
X_k^2 = I_A(W_k)+I_B(W_k).  
\end{align}
Hence, \eqref{ju7} yields
\begin{align}\label{jv7}
  \E\sqpar{X_1^2} =   \E[I_A(W_k)] + \E[I_B(W_k)]
=2q^{-\ell}.
\end{align}
Furthermore, \eqref{ju4} and \eqref{jv5} yield
\begin{align}\label{jv8}
  \sum_{k=1}^{\ell-1}\E\sqpar{X_1X_{1+k}}
=q^{-\ell}\bigpar{\theta_{AA}-\theta_{AB}-\theta_{BA}+\theta_{BB}}.
\end{align}
We conclude from \eqref{gss} and \eqref{jv7}--\eqref{jv8} that
\begin{align}\label{gss9}
\gss
=2q^{-\ell}\bigpar{1+\theta_{AA}-\theta_{AB}-\theta_{BA}+\theta_{BB}}.
\end{align}

For the third cumulant $\gk_3$ we argue similarly.
First, since $\E S_n=0$, we have $\kk_3(S_n)=\E[S_n^3]$, and by expanding
followed by combining equal terms,
\begin{multline}
  \E[S_n^3] 
=\sum_{i,j,k=1}^n \E[X_iX_jX_k]
=\sum_{1\le i\le n}\E[X_i^3]
+3\sum_{1\le i<j\le n}\E[X_i^2X_j]
\\
+3\sum_{1\le i<j\le n}\E[X_iX_j^2]
+6\sum_{1\le i<j<k\le n}\E[X_iX_jX_k].
\end{multline}
Furthermore, all terms in the sums on the \rhs{} with $j\ge i+\ell$ or $k\ge
j+\ell$ vanish by independence, and the terms are invariant under
a simultaneous shift of the indices. Hence \eqref{t1kk3} yields
\begin{align}\label{jw1}
\gk_3
=\E[X_1^3]
+3\sum_{s=1}^{\ell-1}\E[X_1^2X_{1+s}]
+3\sum_{s=1}^{\ell-1}\E[X_1X^2_{1+s}]
+6\sum_{s,t=1}^{\ell-1}\E[X_1X_{1+s}X_{1+s+t}].
\end{align}
We have $X_1^3=X_1$, and thus $\E[X_1^3]=0$.
Furthermore, \eqref{ju4} and \eqref{jv66} together with \eqref{jv5} show that
\begin{align}
  \label{jw2}
\sum_{s=1}^{\ell-1}\E[X_1^2X_{1+s}]
&=q^{-\ell}\bigpar{\theta_{AA}-\theta_{AB}+\theta_{BA}-\theta_{BB}},
\\\label{jw3}
\sum_{s=1}^{\ell-1}\E[X_1X_{1+s}^2]
&=q^{-\ell}\bigpar{\theta_{AA}+\theta_{AB}-\theta_{BA}-\theta_{BB}}.
\end{align}
Arguing as in \eqref{jv5} yields, for three fixed words $U,V,T\in\cA^\ell$,
\begin{align}\label{jw5}
&\sum_{j,k=1}^{\ell-1}  \E\bigsqpar{I_U(W_i)I_V(W_{i+j})I_T(W_{i+j+k})}
\notag\\&\hskip4em
=
\sum_{j,k=1}^{\ell-1}\indic{\ell-j\in\Theta(U,V)}\indic{\ell-k\in\Theta(U,V)}
\,q^{-(\ell+j+k)}
\notag\\&\hskip4em
=q^{-\ell}\theta_{UV}\theta_{VT}.
\end{align}
Hence, \eqref{ju4} yields
\begin{align}\label{jw6}
  \sum_{s,t=1}^{\ell-1}\E[X_1X_{1+s}X_{1+s+t}]&
=q^{-\ell}\bigl(
\theta_{AA}\theta_{AA}-\theta_{AA}\theta_{AB}-\theta_{AB}\theta_{BA}+\theta_{AB}\theta_{BB}
\notag\\[-0.5\baselineskip]&\hskip4em
{}-\theta_{BA}\theta_{AA}+\theta_{BA}\theta_{AB}+\theta_{BB}\theta_{BA}-\theta_{BB}\theta_{BB}
\bigr)
\notag\\&\hskip-2em
=q^{-\ell}(\theta_{AA}-\theta_{BB})(\theta_{AA}+\theta_{BB}-\theta_{AB}-\theta_{BA})
.\end{align}
Finally we obtain from \eqref{jw1} and \eqref{jw2}--\eqref{jw6}, after some
cancellations, 
\begin{align}\label{gk3}
  \gk_3 = 6q^{-\ell}(\theta_{AA}-\theta_{BB})
(1+\theta_{AA}+\theta_{BB}-\theta_{AB}-\theta_{BA})
= 3\gss(\theta_{AA}-\theta_{BB})
.\end{align}

It is also possible to obtain this result from the
  explicit moment formulas. 
Denote by $P_\ell$ the transition matrix, indexed by all $q^\ell$ strings
of length $\ell$. Here, a transition $U\to V$ exists if and only if the
last $\ell-1$ characters of $U$ coincide with the first $\ell-1$
characters of $V$. It is easy to see that $P_\ell$ is primitive and the
stationary distribution is uniform: $\pi=q^{-\ell}\one$; 
see \eqref{cp3} below.
The only remaining difficulty in applying the formulas \eqref{t1a}--\eqref{t1c}
is to calculate
the group inverse $(I-P_\ell)\gx$. It turns out that the entries of this
matrix coincide with the overlap indices $\theta_{UV}$, up to a
normalization. 
\begin{proposition}\label{P4}
The entries of the group inverse $Q:=(I-P_\ell)\gx$ are given by 
\begin{equation}\label{cp1}
(I-P_\ell)\gx_{U,V}=\indic{U=V}+\theta_{UV}-\ell q^{-\ell}
\end{equation}
\end{proposition}
\begin{proof}
Introduce the notation $tail_k(U)$ for the string obtained by deleting the
first $k$ characters of $U$. Analogously for $head_k(U)$. 
It is straightforward 
to see that for any $U,V$ we have 
\begin{align}\label{cp2}
(P^j)_{U,V}& =  q^{-j}\indic{tail_j(U)=head_j(V)},&& 0\leq j<\ell,\\
(P^j)_{U,V}& =  q^{-\ell}, && j\geq \ell.\label{cp3}
\end{align}
Further,
\begin{align}\label{cp4}
    \theta_{UV}& =  \sum_{k=1}^{l-1}q^{k-l}\indic{tail_{l-k}(U)=head_{l-k}(V)}
\notag\\    & 
=  \sum_{j=1}^{l-1}q^{-j}\indic{tail_{j}(U)=head_{j}(V)}
=  \left(\sum_{j=1}^{\ell-1}P^j\right)_{U,V}
.\end{align}
At this point, the result follows 
by a simple calculation
from the formula 
$(I-P)\gx=\lim_{t\upto 1}\bigpar{\sum_{k\geq 0}t^kP^k-\one\pi\tr/(1-t)}$ 
in Proposition \ref{groupinvformulaprop}. 
\end{proof}
In our case, the diagonal matrix $G$ is given by 
$G_{AA}=1$, $G_{BB}=-1$, and $G$ is otherwise zero.

Clearly  $\mu=0$, by \eqref{t1kk1} or \eqref{t1a}.
For $\gss$ and $\gk_3$ we use \eqref{t1e}--\eqref{t1f}, noting that by
\refProp{P4}, $Q'_{U,V}=\theta_{UV}-\ell q^{-\ell}$.
Hence, 
\begin{align}
 \pi\tr G Q' G \one = q^{-\ell} \one\tr G Q' G \one 
= \theta_{AA}+\theta_{BB}-\theta_{AB}-\theta_{BA}.
\end{align}
Furthermore, evidently  $\pi\tr G^2 \one =2 q^{-\ell}$.
Hence, \eqref{t1e} yields
\begin{align}
 q^\ell\sigma^2& =  \one\tr G^2\one+2\one\tr G Q'G\one + 0\\
& =  2+2(\theta_{AA}+\theta_{BB}-\theta_{AB}-\theta_{BA})
\end{align}
in agreement with \eqref{gss9}.
The calculation of the third cumulant \eqref{gk3}
by this method is similar, and the details are omitted.


We now obtain a preliminary  result for Litt's game.
This will be improved to \refT{TAB} in \refS{SLitt2}
where we identify the cases where the condition \eqref{em4} does not hold.
\begin{theorem}\label{TAB0}
Let Alice and Bob play Litt's game above with distinct words $A$ and $B$ of the
same length $\ell$ in an alphabet $\cA$ with $q$ letters, and assume that
$n$ letters are chosen  at random, uniformly and independently.
Assume also that \eqref{em4} holds.
Then, with $\theta_{UV}$ and $\gss$ given by \eqref{tau} and \eqref{gss9},
$\gss>0$ and
\begin{align}\label{tab01}
\P(\text{\rm Alice wins}) &
= \frac12+\frac{\theta_{BB}-\theta_{AA}-1}{2\sqrt{2\pi\gss}}n\qqw+O(n\qw),
\\\label{tab02}
\P(\text{\rm Bob wins}) &
= \frac12+\frac{\theta_{AA}-\theta_{BB}-1}{2\sqrt{2\pi\gss}}n\qqw+O(n\qw),
\\\label{tab03}
\P(\text{\rm Tie}) &
= \frac{1}{\sqrt{2\pi\gss}}n\qqw+O(n\qw),
\end{align}
and thus
\begin{align}\label{tab04}
\P(\text{\rm Alice wins}) 
-
\P(\text{\rm Bob wins}) &
= \frac{\theta_{BB}-\theta_{AA}}{\sqrt{2\pi\gss}}n\qqw+O(n\qw).
\end{align}
\end{theorem}

\begin{proof}
 A consequence of \refC{C1}, 
recalling \eqref{ju5} and using \eqref{gss9} and \eqref{gk3} above.
\end{proof}

\begin{example}[$\sH\sH$ vs $\sH\sT$]
  Litt's original game is the case $\cA=\set{\sH,\sT}$, $q=2$, $\ell=2$,
  $A=\sH\sH$, $B=\sH\sT$ of \refT{TAB}.
We have $\Theta(A,A)=\Theta(A,B)=\set1$ and $\Theta(B,A)=\Theta(B,B)=\emptyset$, and
thus
\begin{align}
  \theta_{AA}=\theta_{AB}=\tfrac12, \quad \theta_{BA}=\theta_{BB}=0.
\end{align}
Hence, \eqref{gss9} yields $\gss=\frac12$, and 
\eqref{tab1}--\eqref{tab4} yield
\eqref{grg1}--\eqref{grg2}, with error terms $O(n\qw)$.
\end{example}

\begin{remark}
  We have here represented the score $\hS_n$ using a (finite-state) Markov
  chain. We may also note that the sequence $X_k$ is $(\ell-1)$-dependent,
which means that general results for sums of $m$-dependent variables can be
applied to $\hS_n$. See e.g.\ \cite{Heinrich1984,Rhee,Loh}
for some related results; however, we have not been able to find a general
result that applies to our situation.
\end{remark}

\section{Proof of \refTs{T1}--\ref{T2}}\label{SpfT1}

\subsection{A lemma}
We will use the following  
simple uniform version of the spectral radius formula \eqref{spec}.
We do not know a reference so we give a proof for completeness.
\begin{lemma}\label{LK}
  Let $z\mapsto A(z)$ be a continuous square-matrix-valued function defined
  on some 
  compact set $K$.
Suppose that for some $r>0$, we have $\rho(A(z))<r$ for every $z\in K$. Then
there exists $\rx<r$ such that
\begin{align}\label{lk}
  \norm{A(z)^n}\le C \rx^n
\end{align}
for some constant $C$, uniformly for all $z\in K$ and $n\ge1$.
\end{lemma}
\begin{proof}
  Let $z\in K$, and choose $r_{z}$ with $\rho(A(z))<r_z<r$. 
Then the assumption and \eqref{spec} show that there exists
  $N$ such that 
$\norm{A(z)^N}<r_{z}^N$.
Since $z\mapsto A(z)$ is continuous, and the operator norm is a continuous
functional, it follows that there exists an open neighbourhood $U_z$ of
$z$ such that for all $w\in U_z$,
\begin{align}\label{jr1}
\norm{A(w)^N}<r_{z}^N.    
\end{align}
Since $\norm{AB}\le\norm{A}\cdot\norm{B}$ for two matrices $A$ and $B$,
it follows that for any $w\in U_z$ and all $n\ge1$, 
by writing $n=kN+\ell$ with $0\le \ell<N$,
\begin{align}\label{jr2}
\norm{A(w)^n}\le \norm{A(w)^N}^k\norm{A(w)^\ell}
\le r_{z}^{kN}\norm{A(w)^\ell}
\le r_{z}^{n}\max_{0\le\ell<N}\bigpar{r_{z}^{-\ell}\norm{A(w)^\ell}}.
\end{align}
For fixed $\ell$, $\norm{A(w)^\ell}$ is a continuous function of $z$, and 
is thus bounded in $K$. Consequently \eqref{jr2} implies that 
\begin{align}\label{jr3}
\norm{A(w)^n}
\le C_z r_{z}^{n}
\end{align}
holds for some $C_z$, uniformly for $w\in U_z$ and $n\ge1$.
We may cover the compact set $K$ by a finite number of such open sets
$U_{z_i}$,
$i=1,\dots,M$;
thus \eqref{lk} follows for all $z\in K$ with $C:=\max_i C_{z_i}$
and $\rx:=\max_i r_{z_i}<r$.
\end{proof}

\subsection{Proof of \refT{T1}}

We assume in this subsection the assumptions of \refT{T1}.

For  complex $z\neq0$, define the matrix
\begin{align}\label{jp6}
  P(z)=(P(z)_{ij})_{i,j\in\cW}
\quad\text{where}\quad
P(z)_{ij}:=P_{ij}z^{g(j)}
\end{align}
and the vector
\begin{align}\label{jp7}
\pi_1(z)=(\pi_1(z)_{i})_{i\in\cW}
\quad\text{where}\quad
\pi_1(z)_{i}:=\pi_{1;i}z^{g(i)}
.\end{align}
Note that $P(1)=P$ and $\pi_1(1)=\pi_1$.
Furthermore, 
the matrix-valued function $P(z)$ and the
vector-valued function $\pi_1(z)$ are 
analytic functions of
$z\neq0$. 

Let 
\begin{align}\label{jp8}
G_n(z):=\E z^{S_n}, 
\end{align}
i.e., the probability generating function of $S_n$
(in a generalized sense since $S_n$ may take both positive and negative integer
values).
Then $G_n(z)$ is a well-defined analytic (in fact, rational) function
for complex $z\neq0$, since $S_n$ takes only a
finite number of values for each $n$.
Note that for real $t$,
\begin{align}\label{lu66}
G_n(e^{\ii t})=\E e^{\ii t S_n}=\gf_{S_n}(t),  
\end{align}
the characteristic function of $S_n$.

We will use the
following well-known representation.
Recall that $\one=(1,\dots,1)$, the (column) vector with all entries 1.
\begin{lemma}\label{L1}
  For any $z\neq0$ and $n\ge1$,
  \begin{align}\label{l1}
G_n(z) 
= \pi_1(z)\tr P(z)^{n-1}\one.
  \end{align}
\end{lemma}
\begin{proof}
Since $\P(W_1=i_1,\dots,W_n=i_n)=\pi_{1;i_1}P_{i_1i_2}P_{i_2i_3}
\dotsm P_{i_{n-1}i_n}$, we have
  \begin{align}\label{l1a}
G_n(z) 
=
\E \prodkn z^{g(W_k)}
= \sum_{i_1,\dots,i_n\in\cW} \pi_{1;i_1} z^{g(i_1)}P_{i_1i_2}z^{g(i_2)} 
\dotsm P_{i_{n-1}i_n}z^{g(i_n)} 
,\end{align}
which equals the \rhs{} of \eqref{l1}.
\end{proof}

We aim at finding good estimates of the \chf{} $\gf_{S_n}(t)=G_n(e^{\ii t})$. 
We consider first small $t$.
The following asymptotic formula is central in our arguments.
\begin{lemma}  \label{L2}
If $\gd\in(0,1)$ is small enough, then there exist analytic functions $\eta(z)$
and $\gl(z)$ in $D_\gd:=D(1,\gd)$ 
and a constant $c\in(0,1)$ such that for  all $z\in D_\gd$ and $n\ge1$ 
\begin{align}\label{l2}
G_n(z) 
=\eta(z)\gl(z)^{n}\bigpar{1+O(c^{n})}
\end{align}
and, somewhat more precisely,
\begin{align}\label{l2b}
G_n(z) 
=\eta(z)\gl(z)^{n}\bigpar{1+O(c^{n}|z-1|)}
.\end{align}
Furthermore, $\gl(z)$ is an eigenvalue of $P(z)$, and $\eta(1)=\gl(1)=1$.
\end{lemma}
\begin{proof}
Denote the set of eigenvalues of $P(z)$ by $\gL(z)$.

Consider first $z=1$. 
The matrix $P(1)=P$ is stochastic, which means that $P\one=\one$, i.e.,
$\one$ is a right eigenvector with eigenvalue 1.
Moreover, by the Perron--Frobenius theorem and our assumptions
\eqref{em0}--\eqref{em1},
this eigenvalue is simple and all other eigenvalues $\gl_i$ of $P$
satisfy $|\gl_i|<1$.
Let
\begin{align}\label{rho'}
  \rho':=\max\set{|\gl|:\gl\in \gL(1)\setminus1}<1.
\end{align}
Let also $\rho_0=(2\rho'+1)/3$ and $\rho_1:=(1-\rho')/3$.
Thus $0<\rho_0<1-\rho_1<1$.
Furthermore, $P(1)$ has exactly one eigenvalue (viz.\ 1) in the open disc
$D(1,\rho_1)$, and all other eigenvalues in the (disjoint) open disc
$D(0,\rho_0)$.
The eigenvalues $\gL(z)$ are the roots of the characteristic polynomial of
$P(z)$, and the coefficients in this polynomial are continuous  (in fact,
analytic) functions of $z$. Hence it follows that there exists a small
$\gd\in(0,1)$ such that if 
$z\in D_\gd$ (i.e., $|z-1|<\gd$), then $P(z)$ has exactly 1 simple
eigenvalue in $D(1,\rho_1)$,  and all other eigenvalues in 
$D(0,\rho_0)$.
For $z\in D_\gd$,
denote the eigenvalue in $D(1,\rho_1)$ by $\gl(z)$.
Thus,
for $z\in D_\gd$,
\begin{align}\label{lu1}
  |\gl(z)-1|<\rho_1\quad\text{and}\quad
|\gl(z)|>1-\rho_1>\tfrac23
.\end{align}
Since $\gl(z)$ is a simple root of the characteristic polynomial of $P(z)$, it
follows from the implicit function theorem that
$\gl(z)$ is an analytic function of $z\in D_\gd$. 
Moreover, 
provided $\gd$ is chosen small enough,
the corresponding left and right eigenvectors $u(z)\tr$ and $v(z)$
can (and will) be normalized by 
\begin{align}\label{jq00}
u(z)\tr\one=1=u(z)\tr v(z),  
\end{align}
for every $z\in D_\gd$, 
and then $u(z)\tr$ and $v(z)$ are analytic functions of $z\in D_\gd$.
Note that, since $P(1)=P$, it follows from \eqref{jp4} and $P\one=\one$  that
$\gl(1)=1$ with normalized eigenvectors
\begin{align}\label{lu2}
  u(1)\tr=\pi\tr
\quad\text{and}\quad
v(1)=\one.
\end{align}

Let $z\in D_\gd$ and let 
\begin{align}\label{jQ}
\QX(z):=v(z)u(z)\tr;   
\end{align}
in other words, the matrix
$\QX(z)$ defines the operator $v\mapsto(u(z)\tr v)v(z)$, which is a
projection onto the eigenspace spanned by $v(z)$; 
in particular, 
\begin{align}\label{Q2}
  \QX(z)^2=\QX(z).
\end{align}
Moreover, $\QX(z)$ commutes
with $P(z)$ and 
\begin{align}\label{jq0}
P(z)\QX(z)=\QX(z)P(z)=\gl(z)\QX(z).   
\end{align}
Consequently, by elementary spectral theory,
\begin{align}\label{jq1}
\tP(z):=P(z)-\gl(z)\QX(z)=(I-\QX(z))P(z)
\end{align}
has the set of eigenvalues
$\gL(z)\setminus\set{\gl(z)}\cup\set0$. This set is contained in $D(0,\rho_0)$
and thus by the definition \eqref{rho}
the spectral radius
\begin{align}\label{jq2}
  \rho(\tP(z))<\rho_0.
\end{align}
Consequently,  the spectral radius formula \eqref{spec} implies that for
some constant $C=C(z)$
\begin{align}\label{jq3}
  \norm{\tP(z)^n}\le C\rho_0^n.
\qquad n\ge1,
\end{align}
By decreasing $\gd$ a little, we may assume that
\eqref{jq3} holds on $\overline{D_\gd}$,
and then, by \refL{LK}, \eqref{jq3} holds uniformly for all  $z\in D_\gd$
with some $C=C(\gd)$.
 
By \eqref{jq0}--\eqref{jq1}, $\tP(z)$ and $\QX(z)$ commute, and
$\tP(z)\QX(z)=\QX(z)\tP(z)=0$. Hence, \eqref{jq1} and \eqref{Q2}
imply
\begin{align}\label{jq4}
  P(z)^n = \bigpar{\gl(z) \QX(z)+\tP(z)}^n
=\gl(z)^n\QX(z)+\tP(z)^n.
\end{align}
Consequently,
\refL{L1} yields,
recalling \eqref{lu1} and \eqref{jq00},
for $z\in D_\gd$,
\begin{align}\label{jq5}
G_n(z)&  
=\pi_1(z)\tr\bigpar{\gl(z)^{n-1}\QX(z)+\tP(z)^{n-1}}\one
\notag\\&
=\gl(z)^{n-1}\pi_1(z)\tr \QX(z)\one +O(\rho_0^{n-1})
\notag\\&
=\gl(z)^{n-1}(\pi_1(z)\tr v(z))(u(z)\tr\one)+O(\rho_0^{n}),
\notag\\&
=\gl(z)^{n-1}\Bigpar{\pi_1(z)\tr v(z)
+O\bigpar{c^n}}
\end{align}
with $c:=\xqfrac{\rho_0}{1-\rho_1}<1$;
note that the $O$ terms 
are uniform for $z\in D_\gd$
since \eqref{jq3} is and $\pi_1(z)$ is bounded for $z\in D_\gd$.
We may, by decreasing $\gd$ if necessary,
assume that $|\pi_1(z)\tr v(z)|>1/2$ for $z\in D_\gd$, and then
\eqref{jq5} yields \eqref{l2} with
\begin{align}
  \label{jq6}
\eta(z):=\gl(z)\qw \pi_1(z)\tr v(z)
.\end{align}
We have noted $\gl(1)=1$, and thus \eqref{jq6} and \eqref{lu2}
yield $\eta(1)=\pi_1\tr \one=1$.

Finally, \eqref{l2} can be written
\begin{align}\label{lu3}
\lrabs{\frac{ G_n(z)}{\eta(z)\gl(z)^n}-1} \le C c^n,
\qquad z\in D_\gd,\; n\ge1.
\end{align}
The function $G_n(z)/\bigpar{\eta(z)\gl(z)^n}-1$ on the \lhs{}
is analytic and vanishes at $z=1$;
hence we can divide it by $z-1$ and obtain, by the maximum principle,
\begin{align}\label{lu4}
\lrabs{\frac{\xqfrac{G_n(z)}{\eta(z)\gl(z)^n}-1}{z-1}} \le C c^n,
\qquad z\in D_\gd,\; n\ge1,
\end{align}
which is \eqref{l2b}.
\end{proof}

We may assume (and actually already have assumed in the proof) that $\gd$ is
so small that $\eta(z)\neq0$ and $\gl(z)\neq0$ in $D_\gd$, and thus
$\log \eta(z)$ and $\log\gl(z)$ are defined there.
Let $\gd_0$ be so small that 
if $z$ is a complex number with 
$|z|<\gd_0$, then
$e^{\ii z}\in D_\gd$. We then can define the analytic functions,
for $|z|<\gd_0$,
\begin{align}\label{lu5}
  \psi(z):=\log \gl\bigpar{e^{\ii z}}
\quad\text{and}\quad
\gam(z):=\log \eta\bigpar{e^{\ii z}}.
\end{align}
We obtain from \eqref{l2b} the estimate
\begin{align}\label{lu6}
  \log G_n(e^{\ii z}) = n\psi(z)+\gam(z)+O\bigpar{|z| c^n },
\qquad |z|<\gd_0.
\end{align}
Note that, recalling \eqref{lu66},  \eqref{lu6}
resembles the elementary decomposition of the characteristic function of the sum
of i.i.d.\ variables as a power of the characteristic function of an
individual variable, but we have here also two ``error terms''.

\begin{lemma}
  \label{L3}
The cumulants of $S_n$ are given by
\begin{align}\label{l3a}
  \kk_m(S_n) = \ii^{-m}\psi^{(m)}(0) n+ O(1)
=\gk_m n + O(1)
\end{align}
for every $m\ge1$,
where
\begin{align}\label{l3b}
  \gk_m:=
\ii^{-m}\psi^{(m)}(0).
\end{align}
(The implicit constant may depend on $m$, but not on $n$.)
In particular,
\begin{align}\label{l3c}
\frac{\kk_m(S_n)}{n}
\to \gk_m
\quad\text{as \ntoo}
.\end{align}
%
\end{lemma}

\begin{proof}
Since the functions in   \eqref{lu6} are analytic, we may differentiate
an arbitrary number of times, and obtain by \eqref{cu2} and Cauchy's estimate,
for every $m\ge1$,
\begin{align}\label{lu7}
  \ii^m\kk_m(S_n) =\frac{\ddx^m}{\ddx t^m}\log G_n(e^{\ii t})\bigr|_{t=0}
=n\psi^{(m)}(0)+\gam^{(m)}(0)+O\bigpar{c^n },
\end{align}
which is a more precise form of \eqref{l3a}--\eqref{l3b}.
\end{proof}

We have so far allowed any initial distribution $\pi_1$, but we now, for
simplicity, assume that the Markov chain $(W_n)\xoo$ is stationary, i.e.,
that the initial distribution $\pi_1=\pi$, and thus $\pi_k=\pi$ for every
$k$ by \eqref{jp3} and \eqref{jp4}. Then the random variables $X_k$ have the
same distribution.
As in \eqref{t1kk1}, denote their mean by
\begin{align}
  \mu:=\E X_k.
\end{align}

\begin{corollary}\label{CL3}
  Suppose that the Markov chain $(W_n)\xoo$ is stationary.
  Then
  \begin{align}\label{cl3}
    \gk_1=\mu,
\qquad
\psi'(0)=\ii\mu,
\qquad
\gam'(0)=0.
  \end{align}
\end{corollary}
\begin{proof}
  Since the random variables $X_k$ have the same distribution,
\eqref{cu4} yields
  \begin{align}\label{lu8}
    \kk_1(S_n)= \E S_n = \sumkn\E X_k = n\mu.
  \end{align}
The result follows from the case $m=1$ of \eqref{l3c} and \eqref{lu7}.
\end{proof}

We are now prepared to prove  the estimate of $\gf_{S_n}(t)$
that we need for small $t$.
\begin{lemma}\label{L4}
  Suppose that the Markov chain $(W_n)\xoo$ is stationary.
Then there exists $\gd>0$ such that, if\/ $|t|\le\gd$, then
\begin{align}\label{l4}
\gf_{S_n}(t) &
=e^{\ii n\mu t -n{\gk_2}t^2/2}
\Bigpar{1-\ii n\frac{\gk_3}{6}t^3}+ O\bigpar{(t^2+n^2t^6)e^{-n\gk_2t^2/4}}
.\end{align}
\end{lemma}
\begin{proof}
Let $\gd<\gd_0$. Then \eqref{lu66}, \eqref{lu6} and Cauchy's estimate yield,
for $|t|\le\gd$,
\begin{align}\label{luc9}
\frac{\ddx^4}{\ddx t^4}\log \gf_{S_n}(t)
=
\frac{\ddx^4}{\ddx t^4}\log G_n(e^{\ii t})  
=
n\psi^{(4)}(t) + \gam^{(4)}(t)+ O\bigpar{c^n}
=O(n).
\end{align}
Consequently, a Taylor expansion as in \eqref{cu3} yields, 
using $\gf_{S_n}(1)=1$, \eqref{lu8} and \eqref{l3a},
for $|t|\le\gd$ and $n\ge1$,
\begin{align}\label{lu9}
  \log \gf_{S_n}(t) 
&=\ii\kk_1(S_n)t-\frac{\kk_2(S_n)}{2}t^2
-\ii\frac{\kk_3(S_n)}{6}t^3+ O(nt^4)
\notag\\&
=\ii n\mu t -n\frac{\gk_2}{2}t^2
-\ii n\frac{\gk_3}{6}t^3+ O(t^2+nt^4).
\end{align}
Hence,
\begin{align}\label{er1}
e^{-\ii n\mu t +n\xfrac{\gk_2 t^2}{2}} \gf_{S_n}(t) 
&=
\exp\Bigpar{-\ii n\frac{\gk_3}{6}t^3+ O(t^2+nt^4)}.
\end{align}
If $\gd$ is small enough, then the real part of the
argument of the exponential function in \eqref{er1}
is less than $n{\gk_2}t^2/4+C$, and consequently, by a Taylor expansion,
\begin{align}\label{er11}
e^{-\ii n\mu t +n\xfrac{\gk_2 t^2}{2}} \gf_{S_n}(t) 
&=
1-\ii n\frac{\gk_3}{6}t^3+ O(t^2+nt^4)
+O\Bigpar{\bigpar{n|t|^3+ t^2}^2e^{n{\gk_2}t^2/4}},
\end{align}
which yields \eqref{l4}, recalling that $|t|=O(1)$ and noting 
$nt^4\le (t^2+n^2t^6)/2$.
%
\end{proof}

We turn to estimating $\gf_{S_n}(t)$ for larger $t$, and again begin by
studying the matrix $P(z)$.
This is where we use our assumption \eqref{em4}.

\begin{lemma}\label{L5}
  If\/ $0<|t|\le\pi$, then
  \begin{align}\label{l5}
\rho\bigpar{P(e^{\ii t})} < 1.    
  \end{align}
\end{lemma}

\begin{proof}
  Let $\gl$ be an eigenvalue of $P(e^{\ii t})$, and let $u\tr=(u_j)_1^n$ be a
  corresponding   left eigenvector. Then
  \begin{align}\label{er2}
    \gl u_k = \sum_{j\in\cW}u_jP(e^{\ii t})_{jk}
= \sum_{j\in\cW}u_jP_{jk}e^{\ii g(k) t}
,\qquad k\in\cW,
\end{align}
and thus, by the triangular inequality,
  \begin{align}\label{er3}
    |\gl|\,| u_k| 
\le \sum_{j\in\cW}|u_j|P_{jk}
,\qquad k\in\cW
.\end{align}
Summing over all $k\in\cW$ yields, since $(P_{jk})$ is a stochastic matrix,
\begin{align}\label{er4}
    |\gl|\sum_{k\in\cW}| u_k| 
\le \sum_{j\in\cW}|u_j| \sum_{k\in\cW}P_{jk}
=\sum_{j\in\cW}|u_j|.
 \end{align}
Consequently, $|\gl|\le1$.

Suppose now, to obtain a contradiction, that $|\gl|=1$.
We then have equality in \eqref{er4}, and thus in \eqref{er3} for every $k$.
Note first that since $(P_{jk})$ is irreducible, 
it follows easily from equality in \eqref{er3} that 
we have $|u_k|>0$ for every $k\in\cW$.
Furthermore, equality when applying the triangle inequality to \eqref{er2}
implies
  \begin{align}\label{er5}
\frac{u_je^{\ii g(k) t}}{\gl u_k}>0
\end{align}
for all $j,k\in\cW$ such that $P_{jk}>0$.
It follows that for any path $\cQ$ given by $i_0,\dots,i_\ell$,
  \begin{align}\label{er6}
\frac{u_{i_0}e^{\ii g(\cQ) t}}{\gl^{\ell(\cQ)} u_{i_\ell}}
=
\prod_{k=1}^\ell\frac{u_{i_{k-1}}e^{\ii g(i_k) t}}{\gl u_{i_k}}
>0
\end{align}
and in particular, for a closed path $\cQ$
  \begin{align}\label{er7}
e^{\ii g(\cQ) t}\gl^{-\ell(\cQ)}
=\frac{e^{\ii g(\cQ) t}}{\gl^\ell(\cQ)}>0.
\end{align}
Furthermore, $|e^{\ii g(\cQ) t}|=|\gl^\ell(\cQ)|=1$, and thus, for every
closed path $\cQ$
\begin{align}\label{er8}
e^{\ii g(\cQ) t}\gl^{-\ell(\cQ)}=1.
\end{align}
The set of all $(k,\ell)\in\bbZ^2$ such that
\begin{align}
  \label{er9}
e^{\ii k t}\gl^{-\ell}=1 
\end{align}
is a subgroup, and thus it follows from
\eqref{em4} and \eqref{er8} that \eqref{er9} holds for all
$(k,\ell)\in\bbZ^2$. Taking $(k,\ell)=(1,0)$, this shows $e^{\ii t}=1$,
which contradicts the assumption on $t$.

This contradiction shows that $|\gl|<1$ for every eigenvalue $\gl$ of
$P(e^{\ii t})$, which is the same as \eqref{l5} by \eqref{rho}.
\end{proof}

\begin{lemma}\label{L6}
  For every $\gd>0$,
there exist $r=r(\gd)<1$ and $C$ such that if\/ 
$\gd\le|t|\le\pi$ and $n\ge1$,
then
\begin{align}
  \label{l6}
\norm{P(e^{\ii t})^n}\le C r^n.
\end{align}
\end{lemma}
\begin{proof}
A consequence of \refLs{L5} and \ref{LK}, taking
$K:=\set{t\in\bbR:\gd\le|t|\le\pi}$.
\end{proof}

\begin{lemma}\label{L7}
  For every $\gd>0$,
there exist $r=r(\gd)<1$ and $C$ such that if\/ 
$\gd\le|t|\le\pi$ and $n\ge1$,
then
\begin{align}
  \label{l7}
\bigabs{\gf_{S_n}(t)}\le C r^n.
\end{align}
\end{lemma}

\begin{proof}
  A consequence of \refLs{L6} and \ref{L1} together with \eqref{lu66}.
\end{proof}

\begin{proof}[Proof of \refT{T1}]
The limits \eqref{t1kk2} and \eqref{t1kk3} exist by \eqref{l3c}; note that
$\gss=\gk_2$. 
We postpone the proof that $\gss>0$ to \refC{CO} in \refS{Sap}.
Define the normalized variable $\tS_n:=(S_n-n\mu)/(\gs\sqrt n)$
which has mean $\E\tS_n=0$ and variance $\Var(\tS_n)=1+O(1/n)$ 
by \eqref{lu8} and \eqref{l3a}.

Using the estimates of $\gf_{S_n}(t)$ in \refLs{L4} and \ref{L7},
the result \eqref{t1}
follows as the classical corresponding result for sums of i.i.d.\ random
variables \cite[Theorem IV.3]{Esseen}.
There are actually two proofs of this theorem in \cite{Esseen}.
We  use the second  proof in \cite[pp.~64--65]{Esseen},
taking there
\begin{align}\label{ess2}
  f_n(t):= \gf_{\tS_n}(t)=
e^{-\ii \mu \sqrt n t/\gs}\gf_{S_n}\Bigpar{\frac{t}{\gs\sqrt n}}
.\end{align}
We further take $T_{3\,n}$ there as $\gd \sqrt n$ with the same $\gd$ as in
\refL{L4}, and we have $d=1$ and $t_0=2\pi$.
A crucial part of the estimate is that \refL{L4} yields
\begin{align}
  \gf_{\tS_n}(t) &
=e^{-t^2/2}
\Bigpar{1-\ii \frac{\gk_3}{6\gs^3\sqrt n}t^3}
+ O\bigpar{n\qw(t^2+t^6)e^{-t^2/4}}
,\qquad |t|\le\gd\gs\sqrt n,
\end{align}
and
thus,
with $g(t):=e^{-t^2/2} \bigpar{1-\ii \frac{\gk_3}{6\gs^3\sqrt n}t^3}$,
\begin{align}\label{ess1}
\int_{-\gd \gs\sqrt n}^{\gd\gs\sqrt n}\lrabs{\frac{\gf_{\tS_n}(t)-g(t)}{t}}\dd t
= O\bigpar{n\qw},
\end{align}
which implies that $I_2'=O(n\qw)$ on \cite[p.~65]{Esseen}.
The rest of the estimates in \cite[pp.~64--65]{Esseen} hold without changes,
and we obtain \eqref{t1}; we omit the details.
We may also use the first proof in \cite[pp.~57--59]{Esseen},
again using \eqref{ess1}; however, this yields
\eqref{t1} with the slightly weaker error term
$O\bigpar{\frac{\log n}n}$ from the estimates of $I_k'$ and $\eps_1$ on
\cite[p.~59]{Esseen}. 

It remains to verify the matrix-based 
cumulant formulas
\eqref{t1a}--\eqref{t1f}.
By Lemma \ref{L3} and \eqref{lu5}, the cumulants are given in terms of the
first three derivatives of $\psi(z)=\log \lambda(e^{iz})$ at $z=0$. By
analyticity of 
$\lambda(e^{iz})$, it is enough to calculate the derivatives of $\log
\lambda(e^{t})$ where $t$ is real. 
This can be done by repeated differentiation of the eigenvalue equation and
some algebra, 
see \refApp{Aev}, but we will here instead
make use of \cite{haviv},
which studies this problem in great generality; in particular,
\cite[Theorem 4.1]{haviv} (with a correction of the sign of $\rho_3$)
provides for $P$ as above and any matrix $E$ and small $t$,
the Taylor series
\begin{align}\label{qa1}
\lambda(P+tE)=1+t\rho_1(E)+t^2\rho_2(E)+t^3\rho_3(E)+O(t^4),
\end{align}
where $\lambda(P+tE)$ denotes the largest eigenvalue,
and $\rho_i$ are the following expressions: 
\begin{align}\label{qa2}
\rho_1(E)&=\pi\tr E\one,\\
\rho_2(E)&=\pi\tr EQE\one,\label{qa3}\\
\rho_3(E)&=\pi\tr EQEQE\one-\rho_1(E)\cdot\pi\tr EQ^2E\one\label{qa4}
.\end{align}
Here, as before, we set $Q:=(I-P)\gx$. 
(Recall that since $P$ is irreducible,
it has a simple largest eigenvalue; by continuity, this holds also for all
small perturbations.)
We will use this result in a different, more general form. 
The expansion \eqref{qa1} holds uniformly for all matrices $E$ in a bounded set.
(This follows e.g.\ since the implicit function theorem implies that the
eigenvalue 
$\gl(P+tE)$ is (for small $t$) an analytic function of the entries in $tE$.)
Consequently, \eqref{qa1} holds also for a continuous matrix-valued function
$E(t)$. Let
$F(t)$ be a smooth one-parameter family of  matrices with
$F(0)=P$ and take $E(t):=(F(t)-F(0))/t$ (with $E(0):=F'(0)$);
then \eqref{qa1} gives an expansion of  $\lambda(F(t))$.
We want to extract the first three derivatives at the origin. For this
purpose, it is sufficient to replace $F(t)$ with its third-order Taylor
series
$P+tF'(0)+\frac12t^2F''(0)+\frac16t^3F'''(0)$, which gives
\begin{align}\label{qa5}
E(t):=F'(0)+\frac12tF''(0)+\frac16t^2F'''(0).  
\end{align}
By substituting \eqref{qa5} in \eqref{qa1}--\eqref{qa4},
we obtain
\begin{align}\label{qa6}
\lambda(F(t))& = 1+t\rho_1(E(t))+t^2\rho_2(E(t))+t^3\rho_3(E(t))+O(t^4)
\notag\\&
=1+m_1t+\tfrac12 m_2 t^2 + \tfrac16 m_3 t^3 + O(t^4),
\end{align}
where, 
collecting terms,
\begin{align}\label{qa7}
m_1&:=[t^1]\lambda(F(t)) =  [t^0]\rho_1(E(t))
=  \pi\tr F'(0)\one,\\
\tfrac12 m_2&:=[t^2]\lambda(F(t)) =  [t^1]\rho_1(E(t))+[t^0]\rho_2(E(t))
\notag\\&
= \tfrac12 \pi\tr F''(0)\one+\pi\tr F'(0)QF'(0)\one,
\label{qa8}\\
\tfrac16m_3&:=\lambda(F(t))[t^3]
 =  [t^2]\rho_1(E(t))+[t^1]\rho_2(E(t))+[t^0]\rho_3(E(t))
\notag\\&
= \tfrac16 \pi\tr F'''(0)\one
+\tfrac12 \pi\tr F'(0)QF''(0)\one+ \tfrac12 \pi\tr F''(0)QF'(0)\one
\notag\\&\qquad
+\pi\tr F'(0)QF'(0)QF'(0)\one
-\pi\tr F'(0)\one\cdot\pi\tr F'(0)Q^2F'(0)\one
.\label{qa9}
\end{align}
In our case, we set $F(t)=P(e^t)$, and it is easy to see from \eqref{jp6} that 
the derivatives are given by $F^{(k)}(0)=PG^k$ where
$G=\diag(g(1),\dots, g(m))$. 
Furthermore, the relation between the raw moments $m_i={\frac {d^i} {dt^i}}\big|_{t=0}\lambda(F(t))$ and cumulants $\gk_i={\frac {d^i} {dt^i}}\big|_{t=0}\log\lambda(F(t))$ is as usual: 
\begin{align}
\mu=\gk_1&=m_1,\label{qb1}\\
\sigma^2=\gk_2& =  m_2-m_1^2,\label{qb2}\\
\gk_3& =  m_3-3m_2m_1+2m_1^3\label{qb3}
.\end{align}
Plugging in \eqref{qa7}--\eqref{qa9}, 
and recalling $\pi\tr P=\pi\tr$,
we get the indicated expressions in
equations \eqref{t1a}--\eqref{t1c}.

Finally, \eqref{gi4} yields
\begin{align}\label{qb4}
  QP=Q-I+\one\pi\tr
=Q'+\one\pi\tr,
\end{align}
and as a  consequence, since $Q\one=0$,
\begin{align}\label{qb5}
Q^2P=Q(QP)=Q(Q'+\one\pi\tr)=QQ'=Q'^2+Q'.  
\end{align}
Substituting these in \eqref{t1b}--\eqref{t1c} yields \eqref{t1e}--\eqref{t1f}.
\end{proof}

\begin{proof}[Proof of \refC{C1}]
We obtain \eqref{c1+} directly from \eqref{t1} by taking $\mu=0$ and $x=0$,
noting that $\gth(0)=1/2$ by \eqref{gth}.
To obtain \eqref{c1-} we instead take $x<0$ and let $x\upto0$, noting
  that $\gth(0-)=-1/2$.

The formula \eqref{c1+-} then follows from
\begin{align}
  \P(S_n<0)-\P(S_n>0) = \P_n(S_n<0)+\P(S_n\le0)-1,
\end{align}
and \eqref{c10} follows from
$ \P(S_n=0) = \P_n(S_n\le0)-\P(S_n<0)$.  
\end{proof}

\subsection{Proof of \refT{T2}}

\begin{proof}[Proof of \refT{T2}]
Let $\hW_k:=(W_{k-1},W_k)$; this is easily seen to be a Markov chain in the
state space
\begin{align}
  \hcW:=\bigset{(i,j)\in\cW^2:P_{ij}>0}.
\end{align}
We can regard $g$ as a function $\hcW\to\bbZ$, and then the definition 
\eqref{jb1} yields the same $S_n$ as \eqref{jp1} for the chain
$(\hW_k)_1^\infty$. 
It is easily seen that the chain
$(\hW_k)_1^\infty$
is stationary, irreducible and aperiodic when
$(W_k)_0^\infty$ is, and the condition \eqref{em4} is the same for both chains.
Hence the result follows from \refT{T1} applied to
$(\hW_k)_1^\infty$.
\end{proof}

\section{The aperiodicity condition and non-degeneracy}\label{Sap}

The general results above use the aperiodicity condition \eqref{em4}.
In this section we give several equivalent forms of it, and also of the
condition $\gss>0$ which, as stated in \refT{T1} is a consequence.
We consider irreducible and aperiodic finite-state Markov chains as in 
\refS{SMarkov}. 
In particular, we specialize to Litt's game, and show that
these conditions are always
satisfied except for the rather trivial cases in \refEs{Egss2},
\ref{EH-T}, and \ref{EHH-TT} below.

\subsection{Bad examples}\label{Sbad}
We begin with a few examples where \eqref{em4} fails, illustrating the
reasons for it and showing the consequences.
The first two examples are degenerate with also $\gss=0$.

\begin{example}\label{Egss1}
Consider Litt's game with 
$A=\sH\sT$ and $B=\sT\sH$. 
It is then obvious by induction that
  \begin{align}\label{egss1}
    \hS_n = \indic{\xi_n=\sT}-\indic{\xi_1=\sT} \in \set{-1,0,1}.
  \end{align}
Thus \eqref{t1kk2} yields $\gss=0$.
Furthermore, it is easily seen that $g(\cQ)=0$ for every closed path $\cQ$,
and thus \eqref{em4} fails.
It is easy to see that the matrix $P(e^{\ii t})$ in \refS{SpfT1} has an
eigenvalue $\gl(e^{\ii t})=1$ for every real $t$, and thus
\refLs{L5}--\ref{L7} fail.
It follows trivially from \eqref{egss1} that for every $n\ge2$,
\begin{align}\label{egss1a}
\P(\text{\rm Alice wins}) =
\P(\text{\rm Bob wins}) =\tfrac14,
\qquad
\P(\text{\rm tie})=\tfrac12.
\end{align}
Hence, \eqref{tab1}--\eqref{tab3} utterly fail, while \eqref{tab4} holds
trivially.
\end{example}

\begin{example}\label{Egss2}
More generally, let $\ell\ge2$ and
consider Litt's game with 
$A=\sH\sT^{\ell-1}$ and $B=\sT^{\ell-1}\sH$. 
Then Alice scores when a run of at least $\ell-1$ tails has begun, and Bob
scores 
when such a run ends, and again it is obvious that
$\hS_n  \in \set{-1,0,1}$. Thus $\gss=0$.
Also, again $g(\cQ)=0$ for every closed path $\cQ$, and  \eqref{em4} fails.
As in \refE{Egss1}, it is easy to see that \eqref{tab1}--\eqref{tab3} fail,
while \eqref{tab4} holds trivially 
by symmetry (reversing the order of the coin tosses, cf.\ \cite{Grimmett}).

We will see below (\refT{Tgss0})
that,
as stated without proof in \cite{Basdevant},
this example and its obvious equivalent variants obtained by
interchanging Alice and Bob 
or $\sH$ and $\sT$ (or  both)
are the only cases in Litt's game where $\gss=0$.
\end{example}

In the following examples,  \eqref{em4} fails but $\gss>0$.
We will see that \refL{L5} fails for them, and as a consequence also
\refL{L6}; not surprisingly also \refL{L7} fails.

\begin{example}\label{EH-T}
  Consider the trivial case of Litt's game with $\ell=1$,
$A=\sH$ and $B=\sT$. 
In this case, $S_n$ is just the number of heads minus the number of tails.
Thus, 
\begin{align}\label{tq1}
S_n=2S'_n-n \quad\text{with}\quad S'_n\in\Bin(n,\tfrac12).
\end{align}
In particular  $S_n\equiv n\pmod2$ and thus $\P(S_n=0)=0$ when $n$
is odd, so \eqref{tab3} fails, and consequently (by symmetry) also
\eqref{tab1}--\eqref{tab2} fail.
(Asymptotic expansions of the probabilities are easily obtained from 
\eqref{tq1}, treating $n$ even and $n$ odd separately; we leave this to the
reader.)

Similarly, $g(\cQ)\equiv\ell(\cQ)\pmod2$ for every path $\cQ$
(closed or not). As a consequence, \eqref{em4} does not hold,
which is the reason why \refT{T1} and \refC{C1} do not apply.
We have $\gss=1>0$ by \eqref{gss9} (or as a consequence of \eqref{tq1}).
In the arguments in \refS{SpfT1}, we have by \eqref{jp6}
\begin{align}\label{tq3}
  P(z)
=\frac12 \begin{pmatrix}
 z&z\qw\\
z&z\qw   
  \end{pmatrix}
.\end{align}
This matrix has rank 1, with eigenvalues $(z+z\qw)/2$ and 0, and it follows
easily from \eqref{lu66} and \eqref{lu1} that
\begin{align}\label{tq5}
  \gf_{S_n}(t) = G_n(e^{\ii t}) 
= \pi_1(e^{\ii t})\tr P(e^{\ii t})^{n-1}\one
= \cos^{n} t,
\end{align}
which of course also follows directly from \eqref{tq1}.
Note that 
$P(-1)$ has an eigenvalue $-1$, which means that 
\refLs{L5} and \ref{L6} do not hold for $t=\pi$;
similarly, \eqref{tq5} shows that \refL{L7} fails for $t=\pi$.
\end{example}

\begin{example}\label{EHH-TT}
  Consider Litt's game with $A=\sH\sH$ and $B=\sT\sT$.
This game is obviously fair by symmetry, 
so \eqref{tab4} holds trivially (with $\theta_{AA}=\theta_{BB}=\frac12$), 
and it is checked below that
\eqref{tab1}--\eqref{tab3} also hold.
Thus the conclusions of \refT{TAB} hold;
nevertheless, \refT{T1} and \refC{C1} do not apply
because \eqref{em4} does not hold.
In fact, if we instead use the alphabet $\cA=\set{0,1}$,
we have, recalling \eqref{ju2}, 
\begin{align}\label{tk1}
g(W_k)=\pm\indic{\xi_k=\xi_{k+1}}
\equiv \xi_k-\xi_{k+1}+1\pmod2,  
\end{align}
and consequently, $g(\cQ)\equiv\ell(\cQ) \pmod2$ for every closed path $\cQ$,
which shows that \eqref{em4} does not hold.

In this case, it is easy to calculate the distribution of $\hS_{n+1}=S_n$
exactly using the arguments in \refS{SpfT1}. We may simplify 
the calculations
by letting
$(W_k)$ be a sequence of i.i.d.\ random bits
with $\P(W_k=0)=\P(W_k=1)=\frac12$, 
which is a trivial Markov chain,
and use the version \eqref{jb1} with
$g(i,j):=\indic{i=j=0}-\indic{i=j=1}$.
It is easily seen that for this version,
\eqref{l1} is modified to
\begin{align}  \label{tk2}
G_n(z)=\pi_0\tr P(z)^n\one
\end{align}
with
\begin{align}\label{tk3}
  P(z):=\bigpar{P_{ij}z^{g(i,j)}}_{i,j\in\cW}
=\frac12
  \begin{pmatrix}
 z&1\phantom{{}\qw}\\
1&z\qw   
  \end{pmatrix}
.\end{align}
This matrix has 
rank 1, with eigenvalues $(z+z\qw)/2$ and 0
(just as \eqref{tq3}), 
and it follows
that, 
for $n\ge1$,
\begin{align}\label{tk4}
  P\xpar{e^{\ii t}}^n =(\cos t)^{n-1} P(e^{\ii t}).
\end{align}
Hence, \eqref{tk2} yields
\begin{align}\label{tk5}
  \gf_{S_n}(t) = G_n(e^{\ii t}) = \frac{1+\cos t}{2}\cos^{n-1} t.
\end{align}
(This factorization has the probabilistic interpretation that
$S_n$ has the same distribution as a sum $\sum_{0}^{n-1}Y_j$ of independent
random variables, where $Y_0+1$ has the binomial distribution $\Bin(2,\frac12)$
and $Y_j=\pm1$ with probability $\frac12$ each for $j\ge1$.)
It follows easily that, for example, with $m:=\floor{n/2}$,
\begin{align}\label{tk6}
  \P(S_n=0) 
= 2^{-2m-1}\binom{2m}{m}
=
\frac{1}{\sqrt{2\pi n}}+O(n^{-3/2}).
\end{align}
We have $\gss=1$, by \eqref{gss9} or directly from \eqref{tk5} which implies
that $\Var S_n = n-\frac12$. Hence \eqref{tab3} holds, and thus (by symmetry)
also \eqref{tab1}--\eqref{tab2}.

This is as expected, but note that the next term in the expansion \eqref{tk6}, 
of order $n^{-3/2}$, will depend on the parity of $n$. The reason is that 
the matrix $P(e^{\ii\pi})=P(-1)$ has an eigenvalue $-1$ on the unit circle;
thus, although the \chf{} \eqref{tk5} vanishes at $t=\pi$, it is not
exponentially small for $t$ close to $\pi$, as it is in \refS{SpfT1} when we
assume \eqref{em4} and as a consequence \refL{L5} holds.
\end{example}

We end with an (artificial) example of a different type of Markov chain where
\eqref{em4} fails. 
\begin{example}\label{E4}
Consider a stationary Markov chain $(W_k)$
with 4 states $\cA=\set{a,b,c,d}$ and the transition
matrix
\begin{align}\label{tj1}
  P=
  \begin{pmatrix}
    0.45&0.45&0.1&0
\\
    0.45&0.45&0.1&0
\\
0&0&0&1
\\
0.5&0.5&0&0
  \end{pmatrix}
.
\end{align}
This is irreducible and aperiodic, with stationary distribution
$(\frac{5}{12},\frac{5}{12},\frac{1}{12},\frac{1}{12})$.
Let $g(a)=1$, $g(b)=-1$, $g(c)=g(d)=0$.
If we partition $\cW$ as $\set{a,b,d}\cup\set{c}$,
then $g(W_k)=0$ if $W_{k-1}$ and $W_k$ are in different parts, but 
$g(W_k)=\pm1$ if $W_{k-1}$ and $W_k$ are in the same part, and it follows
that
\begin{align}\label{tj2}
  g(W_k)\equiv \indic{W_{k}=c} - \indic{W_{k-1}=c}+1
\pmod2.
\end{align}
This implies that for any closed path $\cQ$,
\begin{align}\label{tj3}
  g(\cQ)\equiv \ell(\cQ)
\pmod2,
\end{align}
and thus \eqref{em4} does not hold.
It is easy to see, e.g.\ by the calculations below or by \refProp{PO}, that
$\gss>0$.

It follows also from \eqref{tj2} and induction that
\begin{align}\label{tj4}
  S_n=\sum_{k=1}^n g(W_k) \equiv \indic{W_1=d}+\indic{W_n=c}+n\pmod2.
\end{align}
As \ntoo, $W_1$ and $W_n$ are asymptotically independent, and thus
$\P(W_1=d, W_n=c)\to \frac{1}{12^2}$
and
$\P(W_1\neq d, W_n\neq c)\to \frac{11^2}{12^2}$.
Hence, if $n$ is even, then $S_n$ is even with probability
$\approx\frac{122}{144}=\frac{61}{72}$, while if $n$ is odd then
$S_n$ is even with probability $\approx\frac{11}{72}$.
Hence $S_n$ exhibits a strong periodicity, although the Markov chain itself is
aperiodic.

In the arguments of \refS{SpfT1}, this is reflected in the easily verified
fact that  $P(-1)$ has an eigenvalue $-1$, and thus $\rho(P(-1))=1$ so
\refL{L5} fails for $t=\pi$. \refL{L5} holds for $0<t<\pi$, and thus
\refLs{L6} and \ref{L7} hold for $\gd\le t\le \pi-\gd$.
Furthermore, for $z$ close to $-1$, we have in analogy with \eqref{l2}
an approximation
\begin{align}\label{cr1}
  G_n(z)=\eta_-(z)\gl(z)^n\bigpar{1+O(c^n)}
\end{align}
with $c<1$, where $\eta_-(z)$ and $\gl(z)$ are analytic and $\gl(-1)=-1$;
furthermore, $\gl(z)$ is an eigenvalue of $P(z)$, and we have
$\gl(z)=-\gl(-z)$.

We can now argue as in \refS{SpfT1}, but we have to include terms coming
from $z=-1$. For simplicity, consider 
\begin{align}\label{cr2}
  \P(S_n=0)=\frac1{2\pi}\int_{-\pi}^\pi\gf_{S_n}(t)
=\frac1{2\pi}\int_{-\pi}^\pi G_n(e^{\ii t}).
\end{align}
Calculations show that $\gl'(1)=\gl'(-1)=0$, 
$\gl''(1)=\tfrac56$,
$\gl''(-1)=-\tfrac56$,
$\eta(1)=1$ (as always), and
$\eta(-1)=\frac{25}{36}$;
\eqref{cr2} and standard arguments then yield
\begin{align}
  \P(S_n=0) =
\frac{3}{\sqrt{5\pi n}}\Bigpar{1+(-1)^n\frac{25}{36}+o(1)}
.\end{align}
This shows periodicity for $\P(S_n=0)$, and in particular \eqref{t1}
does not hold.
\end{example}

\subsection{The aperiodicity condition \eqref{em4}}\label{SSem4}

Let $L$ be the subgroup of $\bbZ^2$ generated by the set
$\bigset{(g(\cQ),\ell(\cQ)):\cQ \text{ is a closed path}}$ in \eqref{em4},
and let $G:=\bbZ^2/L$ be the corresponding quotient.
Thus the condition \eqref{em4} says $G=\set1$.

The projection $\Pi_2:(x,y)\mapsto y$ maps $L$ onto the subgroup 
$L_2$ of $\bbZ$
generated
by the set of $\ell(\cQ)$ for all closed paths $\cQ$, and 
since we assume the aperiodicity \eqref{em1}, we have $L_2= \bbZ$;
in other words, $\Pi_2$ is onto $\bbZ$.
Consequently, there exists some $b\in\bbZ$ such that $(b,1)\in L$.
Let 
\begin{align}\label{pj1}
  L_1:=\set{x\in\bbZ:(x,0)\in L}, 
\end{align}
i.e., the kernel of $\Pi_2$
regarded as a subgroup of $\bbZ$.
Then
\begin{align}\label{pj2}
  (x,y)\in L
\iff (x,y)-y(b,1)\in L
\iff x-yb\in L_1
\end{align}
and thus
\begin{align}\label{pj3}
  L=\bigset{(z+yb,y):y\in\bbZ,z\in L_1}
.\end{align}

Similarly, it follows from $(b,1)\in L$ that
every coset $\overline{(x,y)}\in G$ of $L$ has a
representative of the form $(z,0)$.
Consequently, the homomorphism $z\mapsto\overline{(z,0)}$ maps
$\bbZ$ onto $G$, 
and thus $G\cong \bbZ/L_1$.
Hence, $G$ is a cyclic group.
Let $N:=|G|\in\bbN\cup\set\infty$.
In particular, \eqref{em4} holds if and only if $N=1$.
There are two cases:

\begin{romenumerate}
  
\item 
{$L_1=\set0$, $N=\infty$, and $G\cong\bbZ$.}
Then \eqref{pj3} shows that 
$L=\set{y(b,1):y\in\bbZ}$.
The definition of $L$ shows that this holds
if and only if
\begin{align}\label{paj4}
  g(\cQ)=b\ell(\cQ)
\end{align}
for every closed path.
This is a very degenerate case, see \refSS{SSgss0}.

\item 
$L_1=N\bbZ$, $1\le N<\infty$, and $G\cong \bbZ_N$.
It follows  from \eqref{pj3} that $L$ is a two-dimensional lattice
with basis $\set{(N,0),(b,1)}$.
\end{romenumerate}

In both cases, we identify $G$ with $\bbZ_N$ (where $\bbZ_\infty:=\bbZ$)
in the natural way
by the isomorphism $\bbZ/L_1\to G$.  

Every possible step $ij$ with $P_{ij}>0$ in the Markov chain
generates a vector $v_{ij}:=(g(j),1)$ in $\bbZ^2$, and thus a corresponding
element $\gam_{ij}:=\overline{(g(j),1)}\in G$.
If we sum $v_{ij}$ along a closed path, we get (by definition) an element of
$L$, and thus the sum of $\gam_{ij}$ along a closed path vanishes in $G$.
Fix an (arbitrary) element $o\in\cW$, and pick for every $k\in \cW$ a path
$\cQ_k$ from $o$ to $k$. Let $\gam_k\in G$ be the sum of $\gam_{ij}$ along this
path. 
It is easy to see that since the sum along any closed path is 0, the sum
$\gam_k$ does not depend on the choice of $\cQ_k$.
Furthermore, if $P_{ij}>0$, then we may choose $\cQ_j$ as the path $\cQ_i$
followed by $j$, and thus
\begin{align}\label{paj5}
  \gam_j=\gam_i+\gam_{ij},
\qquad\text{if $P_{ij}>0$}.
\end{align}
Note also that since $(b,1)\in L$,
\begin{align}\label{paj6}
  \gam_{ij}=\overline{(g(j),1)} = \overline{(g(j)-b,0)}.
\end{align}
Hence, by identifying $G$ with $\bbZ_N$, we simply have
$\gam_{ij}=g(j)-b$ in $\bbZ_N$, and thus by \eqref{paj5} 
\begin{align}\label{paj7}
  \gam_j=\gam_i+g(j)-b
\qquad\text{in $\bbZ_N$ when $P_{ij}>0$}.
\end{align}


This leads to the following characterizations of \eqref{em4}.

\begin{proposition}\label{PL}
  Let $(W_k)$ be an irreducible and aperiodic Markov chain on a finite state
  space $\cW$, and let $g:\cW\to\bbZ$.
Then, with notation as in \refSs{SMarkov} and \ref{SpfT1}, 
the following are equivalent:
\begin{romenumerate}
\item \label{PL1}
The condition \eqref{em4} does \emph{not} hold.
\item \label{PL2}
$P(e^{\ii t})$ has an eigenvalue $\gl$ with $|\gl|=1$ for some $t\in(0,\pi]$.
\item \label{PL3}
The spectral radius
$\rho\bigpar{P(e^{\ii t})}=1$ for some $t\in(0,\pi]$.
\item \label{PL4}
There exist an integer $N\ge2$, an integer $b$, and integers 
$\gam_i$, $i\in\cW$, such that for every pair $i,j\in\cW$ with $P_{ij}>0$,
\begin{align}\label{pl4}
  \gam_j\equiv \gam_i+g(j)-b \pmod N.
\end{align}
\item \label{PL5}
There exists an integer $N\ge2$, an integer $b$,
and a partition of $\cW$ into
(possibly empty)
sets $\cW_k$, $k=1,\dots,N$, such that if $i\in\cW_k$ and $P_{ij}>0$, then
$j \in \cW_{k+g(j)-b}$, with the index regarded modulo $N$.
\end{romenumerate}
\end{proposition}

\begin{proof}
First, 
the proof of \refL{L5} shows that every eigenvalue $\gl$ of $P(e^{\ii t})$
satisfies $|\gl|\le1$, and thus
\ref{PL2}$\iff$\ref{PL3} by the definition of spectral radius.

Furthermore, \ref{PL4}$\iff$\ref{PL5}, since if \ref{PL4} holds, we may
define $\cW_k:=\set{i:\gam_i\equiv k \pmod N}$,
and conversely we may define $\gam_i=k$ for $i\in\cW_k$.

\refL{L5} shows that if \eqref{em4} holds, then \ref{PL3} does not hold; thus
\ref{PL3}$\implies$\ref{PL1}.

Now suppose that \ref{PL1} holds. This means that in the discussion above, 
$N>1$. If $N<\infty$, then \eqref{paj7} shows the existence of $b$ and
$\gam_i$ such that \eqref{pl4} holds.
If $N=\infty$, then \eqref{pl4} holds in $\bbZ$, and it thus holds for any
integer $N\ge2$.
This shows \ref{PL1}$\implies$\ref{PL4}.

Finally, suppose that \ref{PL4} holds.
Let $t=2\pi/N$ and $\go:=e^{\ii t}=e^{2\pi\ii/N}$, and let $v_i:=\go^{-\gam_i}$.
Then, by \eqref{jp6} and \eqref{pl4} if $P_{ij}>0$, 
and trivially if $P_{ij}=0$,
\begin{align}
  P(\go)_{ij}v_j = P_{ij}\go^{g(j)-\gam_j}
=P_{ij}\go^{b-g(i)}
=P_{ij}\go^bv_i,
\end{align}
and by summing over $j$ we see that $v:=(v_i)_i$ is an eigenvector 
of $P(\go)$ with
eigenvalue $\go^b$. Hence \ref{PL2} holds.
\end{proof}

\begin{example}\label{ENagaev}
Nagaev \cite[Condition C]{Nagaev1961}
makes the (rather strong)
assumption that for every pair $i,j\in\cW$, there exists $k\in\cW$
such that $P_{ik},P_{jk}>0$.
This in combination with \eqref{pl4} implies that 
\begin{align}\label{na1}
  \gam_i \equiv \gam_{k}-g(k)+b\equiv \gam_j \pmod N.
\end{align}
Thus, for every $j$, by choosing $i$ such that $P_{ij}>0$ and applying
\eqref{pl4} again,
\begin{align}\label{na2}
  g(j)\equiv b \pmod N,
\qquad j\in\cW.
\end{align}
Since \cite[Theorem 3]{Nagaev1961} also assumes (in our notation)
that \eqref{na2} does not hold for any $N\ge2$, 
we see that the assumptions in \cite{Nagaev1961} imply that
\refProp{PL}\ref{PL4} cannot hold, 
and thus \refProp{PL} shows that \eqref{em4} holds. 
\end{example}

\subsection{The asymptotic variance $\gss=0$}\label{SSgss0}

Another exceptional case when \refT{TAB} fails is when the asymptotic
variance $\gss=0$. 
This is a very degenerate  case; as shown in \refProp{PO} below, it
happens only if, after subtracting a suitable constant $b$ from $g$, the
sums $S_n$ are deterministically bounded.
We will also see that it cannot happen when \eqref{em4} holds.

\begin{proposition}\label{PO}
  Let $(W_k)$ be an irreducible and aperiodic Markov chain on a finite state
  space $\cW$, and let $g:\cW\to\bbZ$.
Then, with notation as in \refSs{SMarkov} and \ref{SSem4}, 
the following are equivalent:
\begin{romenumerate}
  
\item \label{PO1}
$\gss=0$.
\item \label{PO2}
There exists an integer (or real number) $b$ such that deterministically
\begin{align}\label{po2}
  |S_n-bn| \le C
\end{align}
for some constant $C$, uniformly in $n$.
\item \label{PO3}
There exists an integer (or real number) $b$ such that 
\begin{align}\label{po3}
  g(\cQ)=b\ell(\cQ)
\end{align}
for every closed path $\cQ$.
\item \label{PO4}
There exist an integer $b$ and integers 
$\gam_i$, $i\in\cW$, such that for every pair $i,j\in\cW$ with $P_{ij}>0$,
\begin{align}\label{po4}
  \gam_j= \gam_i+g(j)-b.
\end{align}
\item \label{PO5}
There exists an integer $b$ such that 
\begin{align}\label{po5}
  L=\bbZ(b,1) =\set{(yb,y):y\in\bbZ}.
\end{align}
\item \label{PO6}
The lattice $L$ is one-dimensional.
\item \label{PO7}
$L_1=\set{0}$.
\item \label{PO8}
$N=\infty$.
\item \label{PO9}
$G=\bbZ$.
\end{romenumerate}
Furthermore,
the constant $b$ in \ref{PO2}, \ref{PO3}, \ref{PO4}, and \ref{PO5}
is the same integer.
\end{proposition}

\begin{proof}
The equivalences 
\ref{PO6}$\iff$\ref{PO5}$\iff$\ref{PO7}$\iff$\ref{PO8}$\iff$\ref{PO9}$\iff
$\ref{PO3}
follow from the discussion at the beginning of \refSS{SSem4}, in
particular \eqref{pj3} and the text leading to \eqref{paj4}.
Note that if \eqref{po3} holds for some real number $b$, then
$b\ell(\cQ)\in\bbZ$ for every closed path $\cQ$, and thus
$bx\in\bbZ$ for every $x$ in the group $L_2$ generated by the set of all such
$\ell(\cQ)$. As noted above,
the aperiodicity \eqref{em1} implies $L_2= \bbZ$, and thus $b\in\bbZ$,
so $b$ in \eqref{po3} has to be an integer.

Furthermore, \ref{PO8}$\implies$\ref{PO4} follows by \eqref{paj7} (with
$N=\infty$).

\ref{PO4}$\implies$\ref{PO2}:
It follows from \eqref{po4} that
\begin{align}\label{bj1}
  S_n-bn& 
=\sumkn(g(W_k)-b)
=g(W_1)-b+\sum_{k=2}^n(\gam_{W_{k}}-\gam_{W_{k-1}})
\notag\\&
=g(W_1)-b+\gam_{W_{n}}-\gam_{W_{1}}
\end{align}
and thus \eqref{po2} holds with 
$C=\max_{i\in\cW}|g(i)-b|+2\max_{j\in\cW}|\gam_j|$.
(The same argument yields \ref{PO4}$\implies$\ref{PO3} directly, but we do
not need this.)

\ref{PO2}$\implies$\ref{PO1}:
Obvious by the definition \eqref{t1kk2}.

\ref{PO1}$\implies$\ref{PO3}:
Suppose that \ref{PO3} does not hold. Then there exist two closed paths
$\cQ_1$ and $\cQ_2$ such that
\begin{align}\label{sk1}
g(\cQ_1)/\ell(\cQ_1)\neq g(\cQ_2)/\ell(\cQ_2).  
\end{align}
Let $w_i$ be the starting point of $\cQ_i$. If $w_2\neq w_1$, choose two
paths 
$\cQ_{12}$ and $\cQ_{21}$ from $w_1$ to $w_2$ and from $w_2$ to $w_1$,
respectively. We may then construct the closed paths $\cQ':=\cQ_{12}+\cQ_{21}$
and $\cQ'':=\cQ_{12}+\cQ_2+\cQ_{21}$, both starting at $w_1$, 
by concatenation in the obvious way.
Then
\begin{align}
  g(\cQ'')=g(\cQ_{12})+g(\cQ_2)+g(\cQ_{21})=g(\cQ')+g(\cQ_2)
\end{align}
and similarly $\ell(\cQ'')=\ell(\cQ')+\ell(\cQ_2)$;
hence it follows from \eqref{sk1} that we cannot have
$g(\cQ')/\ell(\cQ')= g(\cQ_1)/\ell(\cQ_1)= g(\cQ'')/\ell(\cQ'')$. 
Consequently, 
by relabelling either $\cQ'$ or $\cQ''$ as $\cQ_2$, we may assume that
\eqref{sk1} holds with the same starting point $w_1$ for both closed paths.

Next, let $\ell_i:=\ell(\cQ_i)\ge1$. Replace $\cQ_i$ by the closed path
$\ell_{3-i}\cQ_i$, i.e., the closed path $\cQ_i$ repeated $\ell_{3-i}$ times.
This replacement does not change $g(\cQ_i)/\ell(\cQ_i)$, and thus
\eqref{sk1} still holds, but now both closed paths have the same length
$\ell_1\ell_2$. 

We may thus assume that 
\eqref{sk1} holds for two closed paths
$\cQ_1$ and $\cQ_2$  
of the same length $\ell\ge1$ and starting from the same point $w_1\in\cW$;
note that this implies $g(\cQ_1)\neq g(\cQ_2)$.
We now fix these paths, and thus $\ell$, 
and define a sequence of stopping times
for the Markov chain by
\begin{align}\label{ck1}
  T_0&:=\min\set{n\ge1:W_n=w_1},
\\\label{ck2}
T_k&:=\min\set{n\ge T_{k-1}+\ell:W_n=w_1},\qquad k\ge1.
\end{align}
Thus, at $T_k$ we return to $w_1$, but we only consider returns after a time
of at least $\ell$.

Since the Markov chain is finite and irreducible, almost surely all $T_k$
are finite. The Markov property shows that the times $T_k$ are renewal times
in the sense that the process starts again at each $T_k$.
Let $\tau_k:=T_k-T_{k-1}$, the waiting time of the $k$:th
``renewal''; then $\tau_1$, $\tau_2$, \dots are i.i.d.\ random variables.
Using the fact that
there exists some $t_0$ such that from every point in $\cW$ we may 
with positive probability reach $w_1$ within $t_0$ steps, 
it is easily seen that the tail probability $\P(\tau_1>t)$ decreases
exponentially as $t\to\infty$; in particular, 
the moment $\E \tau_1^m$ is finite
for every $m\ge1$,
Similarly, $\E T_0^m<\infty$.

Consider the two-dimensional process $(n,S_n)$ and its
``increments'' between the renewal times $T_k$ defined by
\begin{align}\label{sk2}
  \zeta_k:=
\bigpar{T_k-T_{k-1},\set{S_{T_{k-1}+i}-S_{T_{k-1}}:0\le i\le T_k-T_{k-1}}},
\qquad k\ge1,
\end{align}
where the second component thus is a random process defined on an
integer interval of random length. 
The Markov property shows that these increments are i.i.d.
We may thus apply the general result from renewal theory stated in \refL{LR}
below and conclude that, with the notation in \eqref{sk3}--\eqref{sk5},
\eqref{sk6} holds with convergence of all moments. In particular,
\begin{align}\label{sk7}
\frac{\Var[S_n]}{n}=\Var\Bigpar{ \frac{S_n-an}{\sqrt{n}}}
\to \xgss/\xmu
\end{align}
and thus,
recalling \eqref{t1kk2},
\begin{align}\label{sk8}
  \gss=\xgss/\xmu.
\end{align}
Now recall the closed paths $\cQ_1$ and $\cQ_2$ of the same length
$\ell$ constructed above. For each $i=1,2$, with positive probability
the Markov chain starts in $w_1$ and then follows $\cQ_i$ for the next
$\ell$ steps; in this case $T_0=1$ and $T_1=\ell+1$   
so $\tau_1=\ell$, and also $S_{T_1}-S_{T_0}=g(\cQ_i)$.
Consequently, $S_{T_1}-S_{T_0}-aT_1$ takes two different values $g(\cQ_i)-a\ell$
with positive probability, and thus $\xgss>0$ by \eqref{sk5} and
consequently $\gss>0$ by \eqref{sk8}, so \ref{PO1} does not hold.
This completes the proof of \ref{PO1}$\implies$\ref{PO3}, which completes
the chain of equivalences.
\end{proof}

We used above the following  renewal theoretic lemma,
stated here for our Markov chain and $S_n$.
The lemma follows easily from known general results in renewal theory;
we find it convenient to use a version from \cite{SJ376}
that uses a formulation from \cite{Serfozo}.
The powerful idea to analyze Markov chains
(also with an infinite state space)
by a suitable renewal sequence of stopping times is old, and was for example
used 
by \cite{Hipp}.

\begin{lemma}\label{LR}
Let $(T_k)_{k=0}^\infty$ be an increasing sequence of (a.s.\ finite)
random times 
and let $\tau_k:=T_k-T_{k-1}$, $k\ge1$.
Assume that
the increments $\zeta_k$ defined in \eqref{sk2} for $k\ge1$ are 
i.i.d.
(Thus, in particular, the increments $\tau_k$ are i.i.d.)
Assume also that for every $r\ge1$, 
$\E T_0^r<\infty$ and $\E \tau_1^r<\infty$.
Let
\begin{align}\label{sk3}
  a&:=\E[S_{T_1}-S_{T_0}]/\E[\tau_1],
\\\label{sk4}
\xmu&:=\E[\tau_1],
\\\label{sk5}
\xgss&:=\Var[S_{T_1}-S_{T_0}-a\tau_1]
.\end{align}
Then, as \ntoo,
\begin{align}\label{sk6}
\frac{S_n-an}{\sqrt{n}}\dto N\bigpar{0,\xgss/\xmu}
\end{align}
with convergence of all moments.
\end{lemma}
\begin{proof}
Consider first the process only from time $T_0$.
Formally we
define
  \begin{align}
    S_n':=S_{T_0+n}-S_{T_0}
  \end{align}
and note that $(S_n')_{n=0}^\infty$  is a stochastic process that starts at
time $n=0$ with $S'_0=0$;
define also $T_n':=T_n-T_0$.
Then the increments $\zeta_k$ in \eqref{sk2} are the same for $(S_n')$ and
$(T_k')$ as for $(S_n)$ and $(T_k)$. By assumption these are i.i.d.;
this means that
in the terminology of \cite{Serfozo} and \cite{SJ376}, 
the stochastic process $S_n'$ has regenerative increments over the
times  $T'_k$. 

Since the state space $\cW$ is finite, the function $g$ is
bounded, and thus $|S'_{T_1}|\le C (T_1-T_0)=C\tau_1$;
consequently the moment $\E|S'_{T_1}|^r$ is finite
for every $r\ge1$. Moreover, if $M_1:=\sup_{1\le n\le T'_1}|S'_n|$
then $\E M_1^r<\infty$ by the same argument.

We may now apply \cite[Theorems 1.4 and 3.1]{SJ376} and conclude that
\begin{align}\label{sk6'}
\frac{ S'_n-an}{\sqrt{n}}
\dto N\bigpar{0,\xgss/\xmu}
\end{align}
with convergence of all moments.

Finally, 
again since $g$ is bounded,
$|S_{T_0+n}-S_n|\le C T_0$ and $|S_{T_0}|\le C T_0$,
and consequently
\begin{align}\label{ck3}
 | S_n-S_n'|\le |S_{T_0+n}-S_n|+|S_{T_0}|\le C T_0.
\end{align}
It follows that, as \ntoo,
\begin{align}\label{ck4}
\frac{S_n-S_n'}{\sqrt n}\to0
\end{align}
in probability,
which together with \eqref{sk6'} implies \eqref{sk6}. 
Moreover,
since $\E T_0^p<\infty$ for every $p\ge1$, it follows from \eqref{ck3}
that \eqref{ck4} holds
in $L^p$ for every $p\ge1$, which together with the moment convergence in
\eqref{sk6'} implies that all moments converge in \eqref{sk6} too.
\end{proof}

\begin{corollary}\label{CO}
  Let $(W_k)$ be an irreducible and aperiodic Markov chain on a finite state
  space $\cW$, and let $g:\cW\to\bbZ$.
Then, with notation as in \refSs{SMarkov},
if \eqref{em4} holds, then $\gss>0$.
\end{corollary}
\begin{proof}
  \refProp{PO} shows that if $\gss=0$, then \eqref{po3} holds and thus
  \eqref{em4} does not hold.
\end{proof}

This corollary completes the proof of \refT{T1}.

\section{Back to Litt's game}\label{SLitt2}

We consider again Litt's game.
We begin by showing that the only cases where $\gss=0$ are the rather
trivial cases in \refE{Egss2} (including \refE{Egss1}).
This (in the form \eqref{qk1} below)
is stated in \citet{Basdevant} without proof; since we also do not know of a
proof given elsewhere, we give a complete proof.

\begin{theorem}\label{Tgss0}
  For Litt's game in \refS{SAB}, we have $\gss>0$ except in the case
$A=\sH\sT^{\ell-1}$ and $B=\sT^{\ell-1}\sH$ (with $q=2$ and some $\ell\ge2$)
and its variants obtained by interchanging Alice and Bob 
or $\sH$ and $\sT$ (or  both).
\end{theorem}

\begin{proof}
  Suppose that $A$ and $B$ are distinct words such that $\gss=0$.
By \eqref{gss9}, this is equivalent to
\begin{align}\label{qk1}
1+\theta_{AA}-\theta_{AB}-\theta_{BA}+\theta_{BB}=0.
\end{align}

Note first that for any two words $U$ and $V$ of length $\ell$,
\eqref{tau} yields
\begin{align}\label{qk2}
  0 \le \theta_{UV} \le \sum_{k=1}^{\ell-1}q^{k-\ell}
=\frac{1-q^{-(\ell-1)}}{q-1}<\frac{1}{q-1}.
\end{align}
Hence, \eqref{qk1} is impossible if $q\ge3$, and thus we in the rest of the
proof assume $q=2$  and take the alphabet to be $\set{\sH,\sT}$.
Furthermore, we have to have $\ell\ge2$, since  $\theta_{UV}=0$ when $\ell=1$.

Moreover, if $\ell-1\notin\Theta(U,V)$, then similarly
\begin{align}\label{qk3}
\theta_{UV}\le  \sum_{k=1}^{\ell-2}2^{k-\ell} < \frac12.
\end{align}
Hence, \eqref{qk1} cannot hold unless $\ell-1\in\Theta(A,B)$
or
$\ell-1\in\Theta(B,A)$.
By symmetry, we may assume the first, i.e., that the last $\ell-1$ letters
in $A$ are the same as the first $\ell-1$ letters in $B$. 
By symmetry, we may also assume that $A$ ends with $\sT$. This means that
for some word $C$ of length $\ell-2$ and some letters $a,b\in\set{\sH,\sT}$,
we have
\begin{align}\label{qk5}
  A & = aC\sT, \qquad
B = C\sT b.
\end{align}

Recall from \refS{SAB} that 
$\hS_n=S_{n-\ell+1}=\sum_{k=1}^{n-\ell-1} g(\xi_k\dotsm\xi_{k+\ell-1})$ where
$g:\cA\to\bbZ$ is given by $g=I_A-I_B$; thus
\begin{align}\label{qk0}
g(A)=1, \quad g(B)=-1, \quad\text{and otherwise } g=0.
\end{align}

We use \cite[Theorem 2]{SJ286} which shows that then the asymptotic variance
$\gss=0$ if and only if there exists a function $h:\cA^{\ell-1}\to\bbR$
and a constant $\mu$ such that for any word $a_1\dotsm a_\ell$ we have
\begin{align}\label{qk4}
  g(a_1\dotsm a_\ell)
= h(a_2\dotsm a_\ell)
- 
h(a_1\dotsm a_{\ell-1})
+\mu.
\end{align}
(In fact, it follows from \eqref{qk4} that
$\mu=\E g(\xi_1\dotsm \xi_\ell)=\E g(W_1)=0$, and thus
$\mu=0$ is the same as before, see \eqref{ju8}.)

We have $A=aC\sT$ by \eqref{qk5}, and thus $aC\sH\neq A$.
Hence, \eqref{qk0} and \eqref{qk4} yield
\begin{align}\label{qk6}
  0<g(aC\sT)-g(aC\sH)= h(C\sT)-h(C\sH).
\end{align}
Let $\ba$ be the letter in $\set{\sH,\sT}$ distinct from $a$.
Then by \eqref{qk4} again and  \eqref{qk6},
\begin{align}\label{qk7}
g(\ba C\sT)-g(\ba C\sH)= h(C\sT)-h(C\sH)>0.
\end{align}
Since $\ba C\sT\neq A$, \eqref{qk7} is possible only if $\ba C\sH =B$.
Thus, by \eqref{qk5},
\begin{align}\label{qk8}
  \ba C\sH = B = C\sT b.
\end{align}
Hence, 
$b=\sH$ and
\begin{align}\label{qk9}
  \ba C = C\sT .
\end{align}
By counting the number of $\sT$s on both sides, we see that $\ba=\sT$, and
thus $a=\sH$. 
Finally, \eqref{qk9} shows that $C\sT$ is invariant under cyclic
permutations, and thus must contain only one letter; hence $C=\sT^{\ell-2}$.
The conclusion $A=\sH\sT^{\ell-1}$ and $B=\sT^{\ell-1}\sH$ now follows from
\eqref{qk5}. 
\end{proof}

\subsection{The  aperiodicity condition for Litt's game}
\label{verifyaperiodicity}

We next show that for Litt's game, also the aperiodicity condition
\eqref{em4} is  satisfied except 
in the (more or less trivial) 
cases in \refEs{Egss2}, \ref{EHH-TT} and \ref{EH-T},
i.e., in the  exceptional cases in \refT{Tgss0}
and also in the cases $(\sH,\sT)$  and $(\sH\sH,\sT\sT)$ 
(or vice versa).

\begin{theorem}\label{Tem4}
For Litt's game in \refS{SAB}, the aperiodicity condition \eqref{em4}
holds except in the cases (all with $q=2$)
\begin{romenumerate}
\item\label{Tem4a} 
$A=\sH\sT^{\ell-1}$ and $B=\sT^{\ell-1}\sH$ for some $\ell\ge2$
(see \refE{Egss2}),
\item 
$A=\sH$ and $B=\sT$ 
(see \refE{EH-T}),
\item 
$A=\sH\sH$ and $B=\sT\sT$ 
(see \refE{EHH-TT}),
\end{romenumerate}
and their variants obtained by interchanging Alice and Bob 
or $\sH$ and $\sT$ (or  both).
\end{theorem}

\begin{proof}
Recall that a path $\cQ$ is a sequence $i_0,\dots, i_m$ in $\cW=\cA^\ell$, with
each transition 
having a positive probability.  
The length is defined as $m$ and the value $g(\cQ)$ is
$\sum_{k=1}^m g(i_k)$. 
Define the weight of the path as $w(\cQ):=\prod_{k=1}^m P_{i_{k-1},i_k}$;
this is the probability that if the Markov chain starts in $i_0$, it will
follow $\cQ$ for the next $m$ steps.

To satisfy the aperiodicity condition, it is sufficient to show: 
\begin{alphenumerate}
    \item \label{QL1}
 There exists a closed path $\cQ_1$ such that $g(\cQ_1)=0$ and $\ell(\cQ_1)=1$.
    \item \label{QL2}
There exists a  closed path $\cQ_2$ such that $g(\cQ_2)=1$.
\end{alphenumerate}
In this case, there exists some $N\geq 1$ such that $(0,1)$ and $(1,N)$ 
are in the set 
$\bigset{(g(\cQ),\ell(\cQ)):\cQ \text{ is a closed path}}$ in \eqref{em4}
and it is clear that these two vectors together generate $\mathbb{Z}^2$;
thus \eqref{em4} holds.

The first condition \ref{QL1}
is simple, since 
for every $a\in\cA$, $a\cdots a \to a\cdots a$ is a 
closed path in $\cW$ of length 1. 
If $q\ge3$, or $q=2$ and 
at least one of $A$ or $B$ is non-constant, then at
least one of these paths 
also has value 0, thus satisfying \ref{QL1}.
The remaining case $q=2$ with $A=\sH^\ell$ and $B=\sT^\ell$ (or vice versa)
is as noted an exception if $\ell\le 2$; if $\ell\ge3$, \ref{QL1} does not
hold but we note that the sequences of coin tosses obtained by repeating
$\sH\sT$ or $\sH\sH\sT$ yields closed paths $\cQ_2$ and $\cQ_3$ with lengths
2 and 3 and value 0, which is just as good
since $\gcd(2,3)=1$.

To show \ref{QL2}, it is sufficient to show there exists a closed path
through $A$ (of length $\ge1$) 
that avoids $B$. Given such a path, we can truncate
it to the time of first return to $A$ to ensure that the value of the path
is $1$. 

Let $W(m)$ be the total weight of all closed paths from $A$ with length $m$,
that avoid $B$, with $W(0):=1$.  
It is easily seen that these numbers have the following generating
function: 
\begin{align}\label{ql1}
(I-tP_{\neq B})^{-1}_{A,A}=\sum_{m\ge0}W(m)t^m  
\end{align}
where $P_{\neq B}$ is given by zero-ing out the column of $P$ corresponding
to $B$.  

We need to show that at least one $W(m)$ is non-zero for $m>0$. 
We will do this by showing that 
$\sum_{m\ge0}W(m)>W(0)=1$. 
For this we use the formula \eqref{lw2} in \refL{LW} below, which together
with \eqref{ql1} gives
\begin{align}\label{ql2}
  \sum_{m\ge0}W(m)
= (I-P_{\neq B})^{-1}_{A,A}
=Q_{AA}+Q_{BB}-Q_{AB}-Q_{BA}.
\end{align}
(An alternative proof of \eqref{ql2} is given in \refApp{Benv}.)
The explicit formula \eqref{cp1} now yields
\begin{align}\label{ql3}
  \sum_{m\ge0}W(m)
=2+\theta_{AA}+\theta_{BB}-\theta_{AB}-\theta_{BA}
.\end{align}
Hence, by \eqref{ql3} and \eqref{gss9},
\begin{align}\label{ql4}
 \sum_{m\ge0}W(m)>1
\iff
1+\theta_{AA}+\theta_{BB}-\theta_{AB}-\theta_{BA}>0
\iff \gss>0,
\end{align}
which holds by \refT{Tgss0} except in the excluded case \ref{Tem4a}.  
\end{proof}

The proof above used the following lemma. (See also \refApp{Benv}.)

\begin{lemma}\label{LW}
Let\/ $P$ be an irreducible stochastic matrix on a finite state space $\cW$, 
and let $\pi$ be its stationary distribution. Then, for any
$B\in\cW$, the matrix
$(I-P_{\neq B})$ is invertible, and the inverse has the entries, for $i,j\in\cW$,
\begin{align}\label{lw1}
  (I-P_{\neq B})_{ij}\qw  
=Q_{ij}+ \bigpar{Q_{BB} +(I-Q)_{iB}}\frac{\pi_j}{\pi_B}-Q_{Bj}
.\end{align}
In particular, if the stationary distribution $\pi$ is uniform, then
\begin{align}\label{lw2}
  (I-P_{\neq B})_{ij}\qw  
=\indic{i=B}+Q_{ij}+ Q_{BB} -Q_{iB}-Q_{Bj}
.\end{align}
\end{lemma}

\begin{proof}
Recall the identity $(I-P)Q=I-\textbf{1}\pi^t$  \eqref{gi4}. 
Hence, $PQ=Q-I+\textbf{1}\pi^t$. 
Thus, using the Kronecker delta $\gd_{kj}:=\indic{k=j}$, 
\begin{align}
(P_{\neq B}Q)_{kj}& =  \sum_{i\neq B} P_{ki}Q_{ij}
\notag\\&
 =  (PQ)_{kj}-P_{kB}Q_{Bj}
\notag\\&
 =  Q_{kj}-\gd_{kj}-\pi_j-P_{kB}Q_{Bj}.
\end{align}
Similarly 
\begin{align}
\smash{\sum_{i}}
(P_{\neq B})_{ki}(I-Q)_{iB}
&=  (P(I-Q))_{kB}-P_{kB}(I-Q)_{BB}
\notag\\&
=  P_{kB}+(I-Q-\textbf{1}\pi^t)_{kB}-P_{kB}(1-Q_{BB})
\notag\\&
=  P_{kB}+\gd_{kB}-Q_{kB}-\pi_B-P_{kB}+Q_{BB}P_{kB}
\notag\\&
 =  \gd_{kB}-Q_{kB}-\pi_B+Q_{BB}P_{kB}
\end{align}
So if we define $M_{ij}$  to be the right hand side of  \eqref{lw1}, then,
using $\sum_{i\neq B} P_{ki}=1-P_{kB}$, 
\begin{align}
(P_{\neq B}M)_{kj}&
=\sum_i (P_{\neq B})_{ki}\left(Q_{ij}+\left(Q_{BB}+(I-Q)_{iB}\right){\frac {\pi_j}{\pi_B}}-Q_{Bj}\right)
\notag\\&
 =  Q_{kj}-\gd_{kj}+\pi_j-P_{kB}Q_{Bj}
+\bigl((1-P_{kB})Q_{BB}
\notag\\&\qquad
+\delta_{kB}-Q_{kB}
-\pi_B+Q_{BB}P_{kB}\bigr){\frac {\pi_j}{\pi_B}}
-(1-P_{kB})Q_{Bj}
\notag\\&
=  -\delta_{kj}+Q_{kj}+\pi_j-P_{kB}Q_{Bj}
+\left(Q_{BB}+\delta_{kB}-Q_{kB}\right){\frac {\pi_j}{\pi_B}}
\notag\\&\qquad
-\pi_j-Q_{Bj}+P_{kB}Q_{Bj}
\notag\\&
 =  -\gd_{kj}+Q_{kj}+\left(Q_{BB}+(I-Q)_{kB}\right){\frac {\pi_j}{\pi_B}}-Q_{Bj}
\notag\\&
 =  -\gd_{kj}+M_{kj}.
\end{align}
Thus $(I-P_{\neq B})M = I $
and the claim follows. 
\end{proof}
\begin{remark}
An alternative proof of this lemma proceeds by writing $I-P_{\neq
  B}=I-P+uv^t=(I-P+\textbf{1}\pi^t)+(uv^t-\textbf{1}\pi^t)$ for appropriate
vectors $u,v$. Since the matrix $I-P+\textbf{1}\pi^t$ is invertible
(Proposition \ref{groupinvformulaprop}), we can now apply the standard Woodbury
formula to write  $(I-P_{\neq B})^{-1}$ as a rank-2
correction to $(I-P+\textbf{1}\pi^t)^{-1}=Q+\textbf{1}\pi^t$. A similar
argument also works to derive a general formula for rank-$k$ updates to
$I-P$. The details are straightforward and left to the reader.  
\end{remark}
\begin{remark}
\label{rmknvisits}
The quantity $\sum_{m\geq 0}W(m)=(I-P_{\neq B})_{AA}^{-1}$ has the
alternative probablistic interpretation the expected number of visits to $A$
before visiting $B$ for a Markov chain started at $A$ (counting the initial
state as a visit). This follows by considering  $N_i$, 
the expected number of visits
to $A$ before visiting $B$ for a random walk started at $i$, and observing
the linear relation
$N_i=\textbf{1}_{i=A}+\sum_{j\neq B} P_{ij}N_j$.
\end{remark}

\begin{proof}[Proof of \refT{TAB}]
By \refT{Tem4},
except in the excluded cases and in the case $A=\sH\sH$ and $B=\sT\sT$ (or
conversely), 
the condition \eqref{em4} holds 
and then 
the result is given by \refT{TAB0}.

In the remaining case
$A=\sH\sH$ and $B=\sT\sT$, \eqref{em4} fails but 
$\gss>0$ and
\eqref{tab1}--\eqref{tab4}
still hold by simple direct calculations, see \refE{EHH-TT}.
\end{proof}

\section{Analysis of state space expansion}
\label{stateExpPiQ}

The proof of Theorem \ref{T2} involves passing to the expanded state space of pairs $(i,j)$. In order to modify the explicit moment formulas for this case, it is necessary to know how the stationary distribution $\pi$ and group inverse $Q$ are transformed under this state-space-expansion operation. 

The expanded transition matrix $\hP_{ij,i'j'}$ evidently satisfies
\begin{equation}\label{cq1}
\hP_{ij,i'j'}=\indic{j=i'}P_{j,j'}
\end{equation}
We have 
\begin{align}\label{cq2}
\sum_{ij}\pi_iP_{i,j} \hP_{ij,i'j'}&=\sum_{ij}\pi_iP_{i,j} \indic{j=i'}P_{j,j'}\notag\\
& =  \sum_{i}\pi_{i}P_{i,i'}P_{i',j'}\notag\\
& =  \pi_{i'}P_{i',j'}
\end{align}
Therefore the expanded stationary distribution is 
\begin{equation}\label{cq3}
\Pi_{ij}=\pi_iP_{ij}
\end{equation}

We also need to work out the expanded group inverse $\hQ:=(I-\hP)\gx$. 
A simple induction shows 
\begin{eqnarray}\label{cq4}
(\hP^k)_{ij,i'j'}& = & (P^{k-1})_{j,i'}P_{i',j'}
\end{eqnarray}
for $k\geq 1$. Another way to see this is to note that a length-$k$ path in
the extended state space from $(i,j)$ to $(i',j')$ can be decomposed into a
length $k-1$ path in the original state spce between $j$ and $i'$, together
with the transition $i'\to j$.  

Therefore, for $0\le t<1$:
\begin{align}\label{cq5}
(I-t\hP)_{ij,i'j'}^{-1}& =  I_{ij,i'j'}+\sum_{k\geq 1}t^k(\hP^k)_{ij,i'j'}\notag\\
& = I_{ij,i'j'}+\sum_{k\geq 1}t^k (P^{k-1})_{j,i'}P_{i',j'}\notag\\
& =  I_{ij,i'j'}+P_{i',j'}t\sum_{k\geq 0}t^k(P^k)_{j,i'}\notag\\
& =  I_{ij,i'j'}+P_{i',j'}t(I-tP)^{-1}_{j,i'}
.\end{align}
Using the group inverse formula 
in \refProp{groupinvformulaprop} and \eqref{cq3} thus yields: 
\begin{align}
(I-\hP)_{ij,i'j'}\gx& 
=  \lim_{t\upto 1}\bigpar{ (I-t\hP)_{ij,i'j'}^{-1}-\pi_{i'}P_{i'j'}/(1-t)}
\notag\\
& =  \lim_{t\upto1}\bigpar{\indic{i=i',j=j'}+tP_{i',j'}(I-tP)_{j,i'}^{-1}-\pi_{i'}P_{i',j'}/(1-t)}\notag\\
& =   \indic{i=i',j=j'}
+P_{i',j'}\lim_{t\upto1}\left(t(I-tP)_{j,i'}^{-1}-\pi_{i'}/(1-t)\right)
.\label{cq6}
\end{align}
Now: 
\begin{eqnarray}\label{cq7}
t(I-tP)^{-1}-\one\pi\tr/(1-t)& = & t(I-tP)^{-1}-t\one\pi\tr/(1-t)+(-1+t)\one\pi\tr/(1-t)\notag\\
& = & t\bigpar{(I-tP)^{-1}-\one\pi\tr/(1-t)}-\one\pi\tr
.\end{eqnarray}
Thus, using \refProp{groupinvformulaprop} again,
\begin{align}\label{cq8}
\lim_{t\upto1}\bigpar{ t(I-tP)^{-1}-\one\pi\tr/(1-t)}
=(I-P)\gx-\one\pi\tr  
.\end{align}
Returning to the previous calculation in \eqref{cq6}, we thus find by
\eqref{cq8} 
\begin{align}\label{cq9}
\hQ_{ij,i'j'}=
(I-\hP)\gx_{ij,i'j'}& =  \indic{i=i',j=j'}+P_{i',j'}(Q_{j,i'}-\pi_{i'}),
\end{align}
where
$Q=(I-P)\gx$ as before. Consequently, $\hQ':=\hQ-I$ is given by
\begin{align}\label{cq10}
\hQ'_{ij,i'j'}
= P_{i',j'}(Q_{j,i'}-\pi_{i'}),
\end{align}
Now we can directly apply the explicit moment formulas
\eqref{t1a}--\eqref{t1f}
to the expanded state space. 

\section{Further remarks}\label{Sext}

\subsection{Some extensions}

We give here some rather brief comments on extensions of the basic
\refTs{T1} and \ref{T2} for finite-state Markov chains.
We discuss three such extensions separately; they may without difficulty be
combined.

\subsubsection{Higher order expansions}\label{SSShigher}
  We have in our results only considered one-term Edgeworth expansions, with
  an extra term of order $n\qqw$ and error of order $O(n\qw)$.
It is by the methods above possible to obtain also expansions with further terms
and errors of order $n^{-m/2}$ for, in principle, any integer $m$.
This is achieved by taking the Taylor expansion in \eqref{lu9} to higher
order,
and then arguing as above and in \cite{Esseen}. 
Note, however, that then also $\gam(t)$ has to be expanded explicitly in
\eqref{lu9}, and thus the higher order terms in the expansion will also
depend
on the first derivatives $\gl^{(m)}(1)$ of the eigenvalue $\gl(z)$ at $z=1$.
These derivatives can be found e.g.\ by 
repeatedly differentiating the characteristic equation
\begin{align}\label{rhigh1}
  \det\bigpar{P(z)-\gl(z)I}=0
\end{align}
at $z=1$. We leave the details to the reader.
See also \cite{Zeilberger} for asymptotic expansions with higher order
terms for the original $\sH\sH$ vs $\sH\sT$ problem.

\subsubsection{Arbitrary initial values}\label{SSSh}
One generalization of \refT{T1}
is to add to $S_n$ some initial value, say $h(X_1)$ for
a given function $h:\cW\to\bbZ$; thus now
\begin{align}\label{xp1}
  S_n:= h(X_1)+\sumkn g(W_k).
\end{align}
An example where this occurs is if Alice and Bob score words of different
lengths. For a simple example, suppose that Alice scores a point when $\sH\sH$
appears, while Bob scores a point when $\sH\sH\sT$ or $\sT\sH\sT$ appears. 
From the third
coinflip on, this is exactly the original problem in disguise, but here
Alice may also gain a point when the second coin is tossed.
If we as in \refS{SAB} consider the Markov chain consisting of
$W_k$ given by \eqref{ju2}, now taking $\ell=3$, then the net score is
$S_{n-2}$ given by \eqref{xp1} with a suitable $g$ and
$h:=I_{\sH\sH\sH}+I_{\sH\sH\sT}$. 

In this version, \refLs{L1}--\ref{L3} still hold, provided we replace
$\pi_1(z)$ defined in \eqref{jp7} by
\begin{align}\label{jp7h}
\pi^h_1(z)=(\pi^h_1(z)_{i})_{i\in\cW}
\quad\text{where}\quad
\pi^h_1(z)_{i}:=\pi_{1;i}z^{g(i)+h(i)}
.\end{align}
However,
this modification changes also $\eta(z)$ by \eqref{jq6},
and since 
$\gam'(0)=\ii\eta'(1)$ by \eqref{lu5},
$\gam'(0)$ may be modified and is in general no longer 0,
so \refC{CL3} does not hold.
In fact, by  taking derivatives in \eqref{jq6} and comparing
with the original case with $h=0$ 
(for which $\eta'(1)=\gam'(0)=0$ by the proof in \refS{SpfT1}),
we see that
\begin{align}\label{gh1}
  \eta'(1)=0+\sum_{j\in\cW}\pi_{1;j}h(j) = \E h(W_1),
\end{align}
and hence, 
\begin{align}\label{gh2}
  \gam'(0)=\ii \eta'(1)=\ii\E h(W_1).
\end{align}
If we still let $\mu:=\E X_k$, and also 
\begin{align}\label{gh3}
  \gD:=-\ii\gam'(0)=\eta'(1)=\E h(W_1),
\end{align}
then obviously
\begin{align}\label{gh4}
  \E S_n = n\mu + \gD.
\end{align}
In particular, \refL{L3} implies $\gk_1=\mu$.

It is now necessary to include a term
$\gam'(0) t=\ii \gD t$ 
in the Taylor expansion \eqref{lu9}.
The rest of the proof goes through if we replace $n\mu$ by $n\mu+\gD$,
and thus \eqref{t1} holds with this change. We may then replace $x$ by
$x-\gD/(\gs\sqrt n)$, which using Taylor expansions yields
\begin{multline}\label{gh5}
\P(S_n -n\mu\le x\gs\sqrt n)
= \Phi(x) + \frac{\gk_3}{6\gs^3\sqrt{2\pi n}}(1-x^2)e^{-x^2/2}
-\frac{\gD}{\gs\sqrt{2\pi n}}e^{-x^2/2}
\\
+\frac{1}{\gs\sqrt{2\pi n}}\gth(x\gs\sqrt n-n\mu)e^{-x^2/2}
+ O\bigpar{n^{-1}}
.\end{multline}
In other words, \eqref{t1} holds with an extra term 
$-\xqfrac{\gD}{\gs\sqrt{2\pi n}}e^{-x^2/2}$.
We omit the details.
In \refC{C1}, this means an extra term
$-\xqfrac{\gD}{\gs\sqrt{2\pi n}}$ in \eqref{c1+} and \eqref{c1-}, and thus
twice as much in \eqref{c1+-}.

This should not be surprising. Consider for simplicity the case
$\mu=0$ and $x=0$ in \refC{C1}. If $h(W):=1$ deterministically,
so we just add 1 to $S_n$, then $\P(S_n\le0)$ is decreased by $\P(S_n=0)$,
which up to $O(1/n)$  equals
$\xqfrac{1}{\gs\sqrt{2\pi n}}$ by \eqref{c10}.
Hence the result just shown says that if \eqref{xp1} holds so the score on
the average is increased by $\gD=\E h(W_1)$ compared to the standard case, 
then the probability $\P(S_n\le0)$
shifts by $\gD$ times as much as if we instead add 1 
deterministically.

\subsubsection{Arbitrary initial distribution}\label{SSSI}
  We have asssumed in \refTs{T1} and \ref{T2} that the Markov chain is
stationary. We may generalize to an arbitrary initial distribution $\pi_1$
(or, in \refT{T2}, $\pi_0$).
Then  \refLs{L1}--\ref{L3} still hold, with $\gl(z)$ still given by
\eqref{jq6}.
However, as in \refSSS{SSSh},
in general $\gam'(0)\neq0$,
so \refC{CL3} does not hold.
(We still have $\gam'(0)=\ii \eta'(1)$ by \eqref{lu5}, 
and $\eta(z)$ depends on $\pi_1$ by \eqref{jq5}.)
If we define $\mu:=\gk_1$, then
\refL{L3} shows that $\E S_n/n=\kk_1(S_n)/n\to\mu$, but
in general $\E S_n$ is not equal to $n\mu$; in fact, \eqref{lu7} with $m=1$
implies that
\begin{align}\label{rit}
  \E S_n - n\mu = \sumkn(\E X_k -\mu) \to \gD:=-\ii \gam'(0).
\end{align}
As in \refS{SSSh}, 
it is now necessary to include a term
$\gam'(0) t=\ii \gD t$ 
in the Taylor expansion \eqref{lu9}.
This has the same consequences as in \refSSS{SSSh},
and again we have \eqref{gh5} and its consequences.

Again this is not surprising: a different initial distribution implies by 
\eqref{rit} an average extra score $\gD$, and (up to order $n\qw$)
the probabilities in \refT{T1} and \refC{C1} shift by an amount proportional
to this average extra score, just as in \refSSS{SSSh}.

\subsection{Litt's game with several words each}

We observe that our result \refT{T1} can be applied to an even more
general formulation of Litt's game, in which Alice and Bob each have several
words that yield positive points for them 
(this generalization is considered by \cite{Zeilberger});
moreover, we may let different words give different numbers of points. 
As before, the value of each word is encoded in the
function $g$ (with a $\pm$ sign
depending on whether it gives points to Alice or Bob).  

If we can verify \eqref{em4} holds for the chain so defined, then \refT{T1} applies; combining \eqref{t1a}--\eqref{t1f} with \refProp{P4} gives explicit formulas for the asymptotics directly analogous with the two-word case  (although the formulas do not simplify as much in the general case). Verifying \eqref{em4} in this case is however more difficult, since we do not have an analogue of \refT{Tem4}. However we can often verify \eqref{em4} on a case-by-case basis by an adaptation of the proof of \refT{Tem4}. For example, if neither Alice nor Bob gets points for $\sH\dotsm \sH$, and furthermore $[z]\Tr(I-P(z))^{-1}>0$, then \eqref{em4} holds (and hence everything goes through as in the two-word case). In general, it is an interesting question to what extent we can characterize the conditions under which \eqref{em4} holds in the multiple-word case.

\section*{Acknowledgements}
We thank Daniel Litt for posing the problem,
Doron Zeilberger for enthusiasm and bringing us together on the problem,
and Geoffrey Grimmett for inspiring and interesting comments.

\appendix
\section{Taylor expansion of eigenvalues}\label{Aev}
We give here a direct proof of the Taylor expansion
\eqref{qa6}--\eqref{qa9} of the eigenvalue $\gl(F(t))$.
Both the result and the method are old, but we do not know an explicit
reference so we give the details here for completeness. 
It is obvious that the calculations can be continued to the derivative of an
arbitrary order, but we need only the first three derivatives.

\begin{proposition}
  Let $t\mapsto F(t)$ be a $C^3$
(three times continuously differentiable) map 
defined on some neighborhood of $0$ and with values $F(t)$ in the space of
$m\times m$ matrices for some $m\ge1$. Suppose that $F(0)$ has an eigenvalue
$\gl_0$ that is (algebraically) simple.
Let $Q$ be the group inverse of $\gl_0 I-F(0)$ and let
$u\tr$ and $v$ be 
left and right eigenvectors  of $F(0)$ for the eigenvalue $0$,
with the normalization $u\tr v=1$. 
(Such vectors exist since the eigenvalue is simple.)  
Then, for $t$ in a possibly smaller
neighborhood of $0$, $F(t)$ has an eigenvalue $\gl(t)$ such that
$\gl(0)=\gl_0$ and $t\mapsto \gl(t)$ is $C^3$ with,
writing $F',F'',F'''$ for $F'(0),F''(0),F'''(0)$,
\begin{align}
  \gl'(0)&= u\tr F' v, \label{lev1}
\\\label{lev2}
\gl''(0)&
=u\tr F''v+2 u\tr F'Q F'v.
\\\label{lev3}
\gl'''(0)&=
u\tr F''' v +3u\tr F''QF'v +3u\tr F'QF''v
+6u\tr F'QF'QF'v
\notag\\&\qquad
-6(u\tr F'v)(u\tr F'Q^2F'v)
.
\end{align}
\end{proposition}

\begin{proof}
By replacing $F(t)$ by $F(t)-\gl_0 I$ we may assume that $\gl_0=0$.
Then $Q$ is the group inverse of $-F(0)$. (Sorry for the minus sign 
with our choice of  notation!)

The characteristic polynomial $f(\gl;t):=\det(\gl I-F(t))$ has coefficients
that are   $C^3$, and $\gl_0=0$ is a simple root of $f(\cdot;0)$;
thus $f(0,0)=0$ and $\frac{\partial f}{\partial\gl}(0,0)\neq0$.
Hence the implicit function theorem shows the existence of a $C^3$ function
$\gl(t)$ 
that is a simple root of $f(t)$, and thus a simple eigenvalue of
$F(t)$, with $\gl(0)=0$.
It remains only to compute the derivatives at 0. 

It follows, e.g.\ using Cramer's rule, that for small $t$ there
exists an eigenvector $v(t)$ of $F(t)$  such that 
$v(t)$ is a $C^3$ function of $t$ with
\begin{align}
v(0)&=v, \label{ev0}
\\\label{ev1}
  F(t)v(t) &= \gl(t)v(t),
\\\label{ev2}
u\tr v(t)&=1.
\end{align}
We now differentiate the eigenvalue equation \eqref{ev1}, which yields 
\begin{align}
  \label{ev3}
\gl'(t)v(t)+\gl(t)v'(t)=F'(t)v(t)+F(t)v'(t).
\end{align}
Multiply to the left by $u\tr$, and note that repeated differentiating of
\eqref{ev2} yields
\begin{align}
  \label{ev4}
u\tr v^{(k)}(t)=0,\qquad k\ge1.
\end{align}
Hence, \eqref{ev3} yields, using \eqref{ev2} and \eqref{ev4},
\begin{align}
  \label{ev5}
\gl'(t)=
u\tr\bigpar{\gl'(t)v(t)+\gl(t)v'(t)}
=u\tr F'(t)v(t)+u\tr F(t)v'(t).
\end{align}
Taking $t=0$ and recalling $u\tr F(0)=0$ we obtain \eqref{lev1}.

Since $Q$ is the group inverse of $-F(0)$, and $F(0)v=0$, we have
\begin{align}\label{ev6}
  Q v =0.
\end{align}
Moreover, \eqref{ev4} shows that $v'(0)$ is orthogonal to $u$, which 
(since the eigenvalue 0 of $F(0)$ is simple) 
spans the null space of the transpose matrix $F'(0)$; 
hence $v'(0)\in \ran(F(0))$,
which implies (by \eqref{gi1}--\eqref{gi2})
that $v'(0)=-QF(0)v'(0)$. Hence \eqref{ev3} implies, using also
$\gl(0)=0$ and \eqref{ev6},
\begin{align}\label{ev7}
  v'(0)=-Q F(0)v'(0)
=-Q \bigpar{\gl'(0) v -F'(0)v}
=Q F'(0)v.
\end{align}

Differentiate \eqref{ev3} again; this yields
\begin{align}
  \label{ev21}
\gl''(t)v(t)+2\gl'(t)v'(t)+\gl(t)v''(t)
=F''(t)v(t)+2F'(t)v'(t)+F(t)v''(t).
\end{align}
Multiplying to the left by $u\tr$ yields, using \eqref{ev4} and \eqref{ev7},
\begin{align}
  \label{ev22}
\gl''(t)
=u\tr F''(t)v(t)+2 u\tr F'(t)v'(t)+u\tr F(t)v''(t),
\end{align}
which yields \eqref{lev2} by taking $t=0$ and using \eqref{ev7} and \eqref{ev4}.

Furthermore, just as for $v'(0)$, 
\eqref{ev4} shows that $v''(0)$ is orthogonal to $u$, 
and thus $v''(0)=-QF(0)v''(0)$. Hence \eqref{ev21} implies, using 
$\gl(0)=0$, \eqref{ev6},  \eqref{ev7}, and \eqref{lev1},
\begin{align}\label{ev27}
  v''(0)
&=-Q F(0)v''(0)
=-2\gl'(0) Qv'(0)
+Q\bigpar{F''(0)v(0)+2F'(0)v'(0)}
\notag\\&
=-2(u\tr F'(0)v) Q^2F'(0)v
+QF''(0)v+2QF'(0)QF'(0)v
.
\end{align}

Finally, differentiating \eqref{ev22} and taking $t=0$ yields, 
since $u\tr F(0)=0$,
\begin{align}
  \label{ev32}
\gl'''(0)
=u\tr F'''(0)v+3u\tr F''(0)v'(0)+3 u\tr F'(0)v''(0),
\end{align}
and \eqref{lev3} follows by substituting \eqref{ev7} and \eqref{ev27}.
\end{proof}

\section{Alternative proof of \eqref{ql2}} \label{Benv}
We give here an independent and more ``probabilistic" proof 
of \eqref{ql2} in the proof of Theorem \ref{Tem4}; 
this proof bypasses the matrix computations of Lemma \ref{LW}. 

Let $V_A^B$ be the number of times that a random walk started from $A$ hits
$A$ prior to hitting $B$ (we count the initial state as a hit).
By \refR{rmknvisits}, the sum $\sum_{m\ge0}W(m)$ in the proof of \refT{Tem4}
equals the expectation $\E V_A^B$. Recall that
the crux of the proof of Theorem \ref{Tem4} was to show 
that this sum is greater than $W(0)=1$, in other words, that
$\E V_A^B>1$. We showed this by showing the formula \eqref{ql2},
which is a special case of \refL{LW}, 
proved above by algebraic calculations.

We give here an alternative proof 
of the following special case of \refL{LW}, which is sufficient for our use
of it in \eqref{ql2};
the proof uses standard methods and known results
for Markov chains.
(We guess that similar arguments can be used to prove \refL{LW} in general,
but we have not pursued this.)

\begin{lemma}\label{LWB}
Let\/ $P$ be an irreducible stochastic matrix on a finite state space $\cW$, 
and let $\pi$ be its stationary distribution. Then, for any distinct states
$A,B\in\cW$, 
\begin{equation}\label{VQ}
  (I-P_{\neq B})_{AA}\qw  
=
\E V_A^B=Q_{AA}-Q_{BA}+(Q_{BB}-Q_{AB}){\frac {\pi_A}{\pi_B}}
.\end{equation}
\end{lemma}

As a first step in the proof, we introduce the random variable
\begin{equation}
T_A^B:= \textnormal{time that a walk started from $A$ first visits $A$ after having visited $B$}.
\end{equation}

\begin{lemma}
\label{TvsV}
With the notation defined above,
\begin{equation}\label{lb2}
  \E V_A^B
=
\pi_A\E T_A^B 
.\end{equation}
\end{lemma}
\begin{proof}
Let $\gamma$ be a random walk starting from $A$. 
Define the sequence of random times $T_0,T_1,\dots$ as $T_0:=0$ and 
\begin{equation}
T_i:=\min\{j: j>T_{i-1}, \gamma_j=A, \gamma_{j'}=B\textnormal{ for some } T_{i-1}<j'<j\};
\end{equation}
in other words, $T_i$ is the time of the first return to $A$ after having
visited $B$ after $T_{i-1}$.
Since the chain is irreducible and finite, all $T_i$ are finite, almost
surely. Moreover, by the Markov property, all $T_i-T_{i-1}$ are i.i.d. 
By definition, $T_1-T_{0}=T_1=T_A^B$. 

We define associated ``rewards" as 
\begin{equation}
V_i=\#\{j: 0\leq j<T_i: \gamma_j=A\}
\end{equation}
i.e., the number of visits to $A$ before $T_i$. 
Similarly to the above, the differences $V_i-V_{i-1}$ are \iid{} by the
Markov property. 
Slightly less obviously, we have $V_1-V_{0}=V_A^B$. This is because
$T_1$ is the time of the first visit to $A$ which occurs after the first
visit to $B$. Therefore, any visits to $A$ prior to $T_1$ must in fact occur
prior to the first visit to $B$.  

By the law of large numbers we thus have, as \ntoo, a.s.,
\begin{align}
  \frac{T_n}{n} &=\frac{\sum_{i=1}^n(T_i-T_{i-1})}{n}\to \E (T_1-T_0) = \E T_A^B,
\\
  \frac{V_n}{n} &=\frac{\sum_{i=1}^n(V_i-V_{i-1})}{n}\to \E (V_1-V_0) = \E V_A^B,
\end{align}
and thus
\begin{align}\label{ab8}
    \frac{V_n}{T_n} &\to \frac{\E V_A^B}{ \E T_A^B}.
\end{align}
On the other hand, 
by the Markov chain ergodic theorem,
a.s.
\begin{equation}\label{ab9}
\frac{V_n}{T_n}
={\frac {\textnormal{number of visits to $A$ before } T_n}{T_n}}
\to \pi_A.
\end{equation}
Finally, \eqref{lb2} follows by comparing \eqref{ab8} and \eqref{ab9}.
\end{proof}

\begin{proof}[Proof of \refL{LWB}]
The first equality in \eqref{VQ} was noted in \refR{rmknvisits}.
For the second equality, we
introduce the mean first passage times $m_{ij}$. 
For distinct $i$ and $j$,
$m_{ij}$ is the average number of steps for a walk started at $i$ to reach
$j$ for the first time. We define, for convenience, $m_{ii}:=0$.  

We can express $\E T_A^B$ in terms of the mean first passage times. Indeed, waiting for the first visit to $A$ after passing through $B$ is equivalent to waiting for the first visit to $B$, and then subsequently waiting for the first visit to $A$ thereafter. So by linearity of expectation,  
\begin{equation}
\label{mabba}
\E T_A^B=m_{AB}+m_{BA}.
\end{equation}
Combining with Lemma \ref{TvsV}, we thus have 
\begin{equation}\label{vabba}
\E V_A^B=\pi_A (m_{AB}+m_{BA}).
\end{equation}

On the other hand, the mean first passage times are known to be closely related to the entries of the group inverse. More precisely, we have 
\begin{equation}
\label{qm}
Q_{ij}=\pi_j(\tau'_j-m_{ij}),
\qquad i,j\in\cW,
\end{equation}
where $\tau_j':=\sum_{k\neq j}\pi_km_{kj}$ 
\cite[Corollary 11.7]{hunter}. 
As a simple consequence, 
\begin{equation}\label{mijq}
m_{ij}={\frac{Q_{jj}-Q_{ij}}{\pi_j}}.
\end{equation}
Combining \eqref{mijq} with \eqref{vabba} yields 
the second equality in \eqref{VQ}.
\end{proof}

\newcommand\AAP{\emph{Adv. Appl. Probab.} }
\newcommand\JAP{\emph{J. Appl. Probab.} }
\newcommand\JAMS{\emph{J. \AMS} }
\newcommand\MAMS{\emph{Memoirs \AMS} }
\newcommand\PAMS{\emph{Proc. \AMS} }
\newcommand\TAMS{\emph{Trans. \AMS} }
\newcommand\AnnMS{\emph{Ann. Math. Statist.} }
\newcommand\AnnPr{\emph{Ann. Probab.} }
\newcommand\CPC{\emph{Combin. Probab. Comput.} }
\newcommand\JMAA{\emph{J. Math. Anal. Appl.} }
\newcommand\RSA{\emph{Random Structures Algorithms} }
\newcommand\DMTCS{\jour{Discr. Math. Theor. Comput. Sci.} }

\newcommand\AMS{Amer. Math. Soc.}
\newcommand\Springer{Springer-Verlag}
\newcommand\Wiley{Wiley}

\newcommand\vol{\textbf}
\newcommand\jour{\emph}
\newcommand\book{\emph}
\newcommand\inbook{\emph}
\def\no#1#2,{\unskip#2, no. #1,} 
\newcommand\toappear{\unskip, to appear}

\newcommand\arxiv[1]{\texttt{arXiv}:#1}
\newcommand\arXiv{\arxiv}

\newcommand\xand{and }
\renewcommand\xand{\& }

\def\nobibitem#1\par{}

\end{document}